\documentclass[a4paper,10pt]{amsart}
\usepackage[utf8]{inputenc}
\usepackage{amsthm}
\usepackage{amsmath}
\usepackage{bm}
\usepackage{amsfonts}
\usepackage{amssymb}
\usepackage{mathtools}
\usepackage{mathrsfs}
\usepackage{setspace}
\usepackage{xcolor}
\usepackage{textgreek}
\usepackage{thmtools}
\usepackage[pdfdisplaydoctitle,colorlinks,breaklinks,urlcolor=blue,linkcolor=blue,citecolor=blue]{hyperref}

\newcommand{\C}{\mathbb{C}}

\newcommand{\EE}{\mathbb{E}}
\newcommand{\Es}{\mathscr{E}}

\newcommand{\M}{\mathcal{M}}
\newcommand{\N}{\mathbb{N}}

\newcommand{\PP}{\mathbb{P}}

\newcommand{\R}{\mathbb{R}}
\renewcommand{\S}{\mathbb{S}}
\newcommand{\T}{\mathbb{T}}

\newcommand{\Z}{\mathbb{Z}}

\DeclareMathOperator{\diam}{diam}
\DeclareMathOperator{\trace}{Tr}

\renewcommand{\epsilon}{\varepsilon}

\DeclareMathOperator{\imm}{i}

\newcommand{\set}[1]{\left\{#1\right\}}
\newcommand{\pa}[1]{\left(#1\right)}
\newcommand{\bra}[1]{\left[#1\right]}
\newcommand{\abs}[1]{\left|#1\right|}
\newcommand{\norm}[1]{\left\|#1\right\|}
\newcommand{\brak}[1]{\left\langle#1\right\rangle}

\newcommand{\wick}[1]{:\mathrel{#1}:}

\newcommand{\expt}[1]{\mathbb{E}\left[#1\right]}
\newtheorem{thm}{Theorem}[section]

\newtheorem{defi}[thm]{Definition}
\newtheorem{cor}[thm]{Corollary}

\newtheorem{lem}[thm]{Lemma}

\newtheorem{prop}[thm]{Proposition}

\theoremstyle{remark}
\newtheorem{rmk}[thm]{Remark}

\numberwithin{equation}{section}

\title[CLT for 2D Euler Gibbsian Invariant Measures]{A Central Limit Theorem for Gibbsian Invariant Measures of 2D Euler Equations}
\author[F. Grotto]{Francesco Grotto}
  \address{Scuola Normale Superiore, Piazza dei Cavalieri, 7, 56126 Pisa, Italia}
  \email{\href{mailto:francesco.grotto@sns.it}{francesco.grotto@sns.it}}
\author[M. Romito]{Marco Romito}
  \address{Dipartimento di Matematica, Universit\`a di Pisa, Largo Bruno Pontecorvo 5, 56127 Pisa, Italia}
  \email{\href{mailto:marco.romito@unipi.it}{marco.romito@unipi.it}}
  \urladdr{\url{http://people.dm.unipi.it/romito}}
  \keywords{point vortices, central limit theorem, 2D Euler equations}
\date\today

\begin{document}

\begin{abstract}
 We consider Canonical Gibbsian ensembles of Euler point vortices on the 2-dimensional torus or in a bounded domain of $\R^2$.
 We prove that under the Central Limit scaling of vortices intensities, and provided that the system has zero global space average
 in the bounded domain case (neutrality condition),
 the ensemble converges to the so-called Energy-Enstrophy Gaussian random distributions. This can be interpreted as describing
 Gaussian fluctuations around the mean field limit of vortices ensembles of \cite{clmp92,KieWan2012}, and it generalises
 the result on fluctuations of \cite{bodineau}.
 The main argument consists in proving convergence of partition functions of vortices and Gaussian distributions.
\end{abstract}

\maketitle

\section{Introduction}

The close resemblance between Onsager's point vortices ensembles and Energy-Enstrophy Gaussian invariant
measures for the two dimensional Euler flow is known since the works of Kraichnan on two-dimensional turbulence, \cite{kraichnan}.
In the present paper, we rigorously establish this connection, as we now outline.
On a two dimensional domain $D$, which in the following will be the two dimensional torus $\T^2$ or a bounded domain of $\R^2$,
Euler equations in vorticity form are given by
\begin{equation*}
	\begin{cases}
	\partial_t \omega + u\cdot \nabla \omega =0,\\
	\nabla\cdot u=0,\\
	\nabla^\perp \cdot u=\omega,
	\end{cases}
\end{equation*}
where $\nabla^\perp=(\partial_2,-\partial_1)$.
Since $u$ is a divergence-less vector field in dimension 2, it can be expressed as $u=\nabla^\perp \phi$;
the \emph{stream function} $\phi$ then must satisfy $\Delta \phi= \omega$, and one can thus recover the velocity field 
from vorticity by $u=\nabla^\perp \phi=-\nabla^\perp (-\Delta)^{-1}\omega$.
The equations have to be complemented with a gauge choice, that is null space average on $\T^2$ and Dirichlet boundary conditions
on $\phi$ in the case $D\subset \R^2$. 
Euler equations are known to be well posed for initial data $\omega_0\in L^\infty(D)$ (see \cite{marchioropulvirenti}),
and smooth solutions preserve the first integrals \emph{energy} and \emph{enstrophy},
\begin{equation}\label{firstintegrals}
	E=\int_D |u|^2dx, \qquad S=\int_D \omega^2dx.
\end{equation}
The Gaussian field associated to the quadratic form $\beta E+\gamma S$ on $\T^2$, the \emph{energy-enstrophy measure} formally defined as
\begin{equation}\label{energyenstrophymeasure}
	d\mu_{\beta,\gamma}(\omega)=\frac1{Z_{\beta,\gamma}}e^{-\beta E(\omega)-\gamma S(\omega)}d\omega,
\end{equation}
is thus a natural candidate as an invariant measure of the flow. However, the field is only supported on spaces of quite rough 
distributions --not even measures-- so that making sense of Euler equations in this setting is not trivial:
this problem has been effectively tackled both by means of Fourier analysis, see for instance \cite{ardfhk,albeveriocruzeiro},
and approximation by point vortices systems, \cite{flandoli,flandoliluo17,flandoliluo19,grotto}. 
The latter ones are defined, let us say first on $\T^2$, as systems of $N$ point particles with positions $x_i\in D$ and intensities $\xi_i\in\R$,
satisfying the system of ordinary differential equations
\begin{equation*}
	\dot x_{i,t}=-\sum_{j\neq i} \xi_j \nabla^\perp G(x_{i,t},x_{j,t}),
\end{equation*}
where the interacting potential is given in terms of the Green function $G$ of the Laplace operator $-\Delta$,
subject to the aforementioned boundary conditions.
The vorticity distribution $\omega=\sum \xi_i \delta_{x_i}$ solves Euler equations in weak sense, see \cite{marchioropulvirenti}:
indeed, it is driven by the vector field $u=\nabla^\perp G\ast \omega$, which, as already noted above, is the equivalent Biot-Savart formulation of 
$\omega=\nabla^\perp\cdot u$.
The system is Hamiltonian with respect to the conjugate coordinates $(\xi_i x_{i,2},x_{i,1})$, and Hamiltonian function
\begin{equation*}
H(x_1,\dots,x_n)=\sum_{i< j}^N \xi_i\xi_j G(x_i,x_j),
\end{equation*}
that is the interaction energy of the vortices. On a bounded domain, the presence of an impermeable boundary produces self interaction terms,
which have to be added to the Hamiltonian in order for the system to satisfy (in weak sense) Euler dynamics 
(see \cite[Section 4.1]{marchioropulvirenti}). 
In both cases, notwithstanding the singularity of the interaction potential,
a slight modification of the arguments in \cite{marchioropulvirenti} -which are set on the whole $\R^2$- 
shows that the system is well-posed for almost every initial
condition $(x_i,\xi_i)_{i=1,\dots N}$ with respect to product Lebesgue measure, the latter being preserved according to Liouville theorem.
Euler point vortices also preserve the \emph{canonical Gibbs ensemble} at inverse temperature $\beta\geq 0$,
\begin{equation*}
	\nu_{\beta,N}(dx_1,\dots,dx_n)= \frac{1}{Z_{\beta,N}} \exp\pa{-\beta H(x_1,\dots,x_n)}dx_1,\dots,dx_n.
\end{equation*}
This measure was first introduced by Onsager in this context, \cite{onsager}.
Equilibrium ensembles at high kinetic energy, which exhibit the tendency to cluster vortices of same sign intensities expected in a turbulent regime,
were proposed by Onsager allowing negative values of $\beta$. Unfortunately, we will not be able to treat the case $\beta<0$ with our arguments.

As our main result, we obtain the Gaussian energy-enstrophy measure as a limit of Gibbsian point vortices ensembles,
in a sort of Central Limit Theorem. Namely, we will consider increasingly many vortices sending $N\rightarrow\infty$,
while decreasing their intensities $\xi_i=\frac{\sigma_i}{\sqrt{\gamma N}}$, with $\gamma>0$ and $\sigma_i=\pm 1$,
as in the familiar central limit scaling. 
We will prove that, if positions of vortices $x_1,\dots,x_N$ have joint distribution $\nu_{\beta,N}$ on $\T^{2N}$, the random measure
\begin{equation*}
	\sum_{i=1}^N \xi_i (\delta_{x_i}-1)\xrightarrow{N\rightarrow\infty} \mu_{\beta,\gamma}
\end{equation*}
converges in law to the energy-enstrophy measure. On $\T^2$, the result does not depend on the choice of signs $\sigma_i$:
to each Dirac delta representing a vortex we are subtracting its space average, so that the global average vanishes
and we are thus looking at fluctuations around a null profile.
In fact, our result can be regarded as an investigation of Gaussian fluctuations
around the well-known mean-field limit, in the case where the latter vanishes, see \autoref{sec:meanfield} below.
This is the reason why we will need to impose (asymptotic) neutrality of the global intensity on bounded domains $D$,
that is, to ensure that the limit in the law of large numbers scaling is naught, since in that case it is not possible
to renormalise Dirac deltas because of the boundary condition.

Most of the underlying physical understanding of the topic goes back to classical works: we mainly refer to the ones of Kraichnan and Onsager,
see respectively \cite{kraichnan,onsager} and references therein. 
The monography \cite{marchioropulvirenti} covers the basic theory of point vortices systems, especially in its dynamical aspects.
We mostly refer to \cite{clmp92,clmp95,lionsbook} and related works for the statistical mechanics of equilibrium ensembles of point vortices.
Our result in a sense completes the one of \cite{bpp}, in which the same scaling limit of point vortices was performed, but with
a smoothed interaction potential. We also mention that a Central Limit Theorem for fluctuations
of point vortices in the case where $D$ is a disk was derived at the end of \cite{bodineau}: 
that result is unfortunately incomplete, since it proves convergence of
integrals of the fluctuation field against a restricted set of test functions. Both \cite{bpp,bodineau} emphasise the relevance
of a good control of partition functions, which in fact is crucial in the present work.
Finally, we mention the Central Limit Theorem of \cite{serfaty18}, concerning a different, \emph{mesoscopic} scaling regime.

\subsection{General Outline and Notation}

In \autoref{sec:clttorus} we discuss in detail our main result in the case where $D=\T^2$ is the 2-dimensional torus.
First, rigorous definitions and properties of Gibbsian ensembles of point vortices and Gaussian invariant measures of Euler equation
are recalled. As already mentioned, the core argument is a uniform bound for partition functions of canonical Gibbs measures,
the strategy being the following:
\begin{itemize}
	\item we split the interaction potential, the Laplacian Green function, into a regular, \emph{long range} part and a singular,
	\emph{short range} part,
	the latter being the Green function of the operator $m^2-\Delta$ (\emph{2-dimensional Yukawa potential}); 
	\item the contribution of the regular part can be interpreted as an exponential integral of a regular Gaussian field:
	since the covariance kernel corresponds to a fourth order operator, no normal ordering is required;
	\item on the other hand, the contribution of the (pointwise vanishing) singular part is controlled by estimating
	the partition function of vortices interacting by Yukawa potential with diverging mass $m\rightarrow\infty$.
\end{itemize}
\autoref{thm:clttorus} is the main result of \autoref{sec:clttorus}. In principle, it could be extended to compact Riemannian surfaces $D$:
we do not pursue such generality, and we only consider two other physically relevant geometries, namely the 2-dimensional
sphere $\S^2$ and bounded domains of $\R^2$. 
The former, being a compact surface without boundary, is completely analogous to the case on $\T^2$, and it is briefly discussed in \autoref{sec:sphere}.
In \autoref{sec:cltdomain} we show how to adapt the previous arguments to the case of a bounded domain,
the main issue being the self-interaction terms in the Hamiltonian due to the presence of a boundary.
Finally in \autoref{sec:meanfield}, as concluding remarks, we outline how our result compares to the well established literature
on mean field limits for point vortices. 

Throughout the paper, the symbols $\simeq, \lesssim$ denote (in)equalities up to uniform multiplicative factors.
The symbol $\sim$ denotes equality in law of random variables.
The letter $C$ denotes possibly different constants, depending only on its eventual subscripts. Finally, $\chi_A$
is the indicator function of the set $A$.

\section{The Periodic Case}\label{sec:clttorus}

Let $\T^2=\R^2/\Z^2$ be the 2-dimensional torus, and denote $d(x,y)$ the distance between two points $x,y\in\T^2$.
We work in the zero average setting, that is we only consider functions (or distributions) having
zero average on $\T^2$: 
we keep it in mind denoting with $\dot L^p(\T^2), \dot H^\alpha(\T^2)$ Lebesgue and Sobolev spaces of zero averaged functions.
It will be convenient to work with Fourier series: let $e_k(x)=e^{2\pi\imm k\cdot x}$, 
for $k\in\Z^2_0=\Z^2\setminus\set{0}$, $x\in\T^2$, be the orthonormal basis of $\dot L^2(\T^2)$ diagonalising the Laplace operator,
and recall that Sobolev spaces (of zero average distributions) are characterised as follows:
\begin{equation*}
	\forall \alpha\in\R, \quad \dot H^\alpha(\T^2)=
	\set{u\in C^\infty(\T^2)': \norm{u}_{\dot H^\alpha}^2=\sum_{k\in\Z^2_0} |k|^{-2\alpha} |\hat u_k|^2<\infty},
\end{equation*}
where $\hat u_k=\brak{u,e_k}$, the brackets denoting (complex) $\dot L^2$-based duality couplings from now on.
We will also denote by $\M(\T^2)$ the linear space of finite signed measures on $\T^2$, which is
continuously embedded in $H^\alpha(\T^2)$ for any $\alpha<-1$, since Fourier coefficients of measures
are uniformly bounded by 1.

The Green function of the Laplace operator with zero average, $G=(-\Delta)^{-1}$, is the unique solution of
\begin{equation*}
\forall x,y\in\T^2 \quad -\Delta_x G(x,y)=\delta_y(x)-1, \quad \int_{\T^2} G(x,y)dx=0;
\end{equation*}
we recall that $G$ is a symmetric function, and moreover it is translation invariant.
It has the explicit representation in Fourier series
\begin{equation*}
	G(x,y)=G(x-y)=\sum_{k\in\Z^2_0} \frac{e_k(x-y)}{4\pi^2|k|^2},
\end{equation*}
and moreover it can be expressed as the sum of Green's function on the whole plane and a bounded function,
\begin{equation}\label{greentorus}
    G(x,y)=-\frac{1}{2\pi} \log d(x,y) +g(x,y),
\end{equation}
with $g(x,y)\in C^0_{sym}(\T^{2\times 2})$. The latter representation holds more generally on any
compact Riemannian surface without boundary (see \cite{aubin}), and it
can be recovered comparing the $G(x,y)$ to the solution of $-\Delta_x u(x)=\delta_y(x)$ on a
small ball centred in $y$ with Dirichlet boundary conditions.

\subsection{Canonical Gibbs Ensembles of Point Vortices}\label{ssec:gibbsensembles}

We now define a Gibbsian canonical ensemble for point vortices distributions of vorticity. 
Let $N\in\N$ (the number of vortices), $\gamma>0$, $\beta\geq 0$ (the \emph{inverse temperature}), 
$\xi_1,\dots,\xi_N\in\R$ (the intensities of vortices), $x_1,\dots,x_N\in\T^2$ (the positions of vortices) and the Hamiltonian
\begin{equation*}
	H(x_1,\dots,x_N)=\sum_{i< j}^N \xi_i\xi_j G(x_i,x_j)
\end{equation*}
on the phase space $\T^{2\times N}$. In what follows, intensities will always be given as 
$\xi_i=\frac{\sigma_i}{\sqrt{\gamma N}}$, with signs $\sigma_i=\pm1$, according to the central limit scaling.
The arguments of the present Section works for any choice of the sequence of signs
$\sigma_1^N,\dots \sigma_N^N=\pm 1$ for $N\geq 1$: we assume that such a choice is performed once and for all,
and drop the apex $N$ to ease notation.
Let us consider the measure on $\T^{2\times N}$ defined by
\begin{equation}\label{gibbsmeasure}
	\nu_{\beta,\gamma,N}(dx_1,\dots,dx_N)= \frac{1}{Z_{\beta,\gamma,N}} \exp\pa{-\beta H(x_1,\dots,x_N)}dx_1,\dots,dx_N,
\end{equation}
with $Z_{\beta,\gamma,N}$, the \emph{partition function}, being the constant such that $\nu_{\beta,\gamma,N}$ 
is a probability measure. 
Notice that, even if it is not made explicit, the partition function depends also on the choice of signs $\sigma_i$.
The measure $\nu_{\beta,\gamma,N}$ is usually referred to as the \emph{canonical Gibbs' measure}.
Since the potential $G$ has a logarithmic singularity, the existence of such measure,
or equivalently the finiteness of $Z_{\beta,\gamma,N}$, is not completely trivial.
For the sake of completeness, and since we could not find a reference matching our setting,
we report the proof. The issue is addressed in \cite{lionsbook} on bounded domains of $\R^2$ for vortices with equal intensities.
The technique we apply was first introduced in \cite{deutschlavaud} in the similar case of a log-gas:
a more refined computation deriving the asymptotics in $N$ in the latter setting can be found in \cite{gunsonpanta}.

\begin{prop}\label{prop:zgreentorus}
For any choice of $\gamma>0$, $\beta\geq0$, and signs $\sigma_i=\pm 1$ as above, if $N>\frac{\beta}{\pi\gamma}$
then $Z_{\beta,\gamma,N}<\infty$, and the measure $\nu_{\beta,\gamma,N}$ is thus well-defined.	
\end{prop}

\begin{proof}
 By (\ref{greentorus}) and H\"older's inequality,
 \begin{equation*}
 	Z_{\beta,\gamma,N} \leq \pa{\int_{\T^{2N}} \prod_{i<j} d(x_i,x_j)^{\frac{\beta\xi_i\xi_j}{\pi}}}^{1/2}
 	\pa{\int_{\T^{2N}} \prod_{i<j} e^{-2\beta\xi_i\xi_jg(x_i,x_j)}}^{1/2},
 \end{equation*}
 where the second factor on the right-hand side is bounded (by a constant depending on all parameters including $N$)
 since $g$ is. Let us now turn to the first term. We relabel the variables as follows: $y_1,\dots y_k$ are the
 ones with positive intensities, and $z_1,\dots z_{n-k}$ the negative ones; moreover, $y_i$ and $z_i$ are
 couples of closest positive-negative neighbours, so that
 \begin{equation}\label{minimaldipoles}
 	d(y_i,z_i)\leq d(y_i,z_j)\wedge d(y_j,z_i) \quad \forall j\geq i.
 \end{equation}
 We accordingly split
 \begin{equation*}
 	\prod_{i<j} d(x_i,x_j)^{\frac{\beta \sigma_i \sigma_j}{\pi \gamma N}}
    =\pa{\frac{\prod_{i<j} d(y_i,y_j)
    	\prod_{i<j} d(z_i,z_j)}
        {\prod_{i,j} d(y_i,z_j)}}^{\frac{\beta}{\pi \gamma N}},
    \end{equation*}
  the indices running over all admissible values. By definition and the triangular inequality,
  \begin{align*}
  	d(y_i,y_j)&\leq d(y_i,z_i)+d(y_j,z_i)\leq 2d(y_j,z_i),\\
  	d(z_i,z_j)&\leq d(y_i,z_i)+d(y_i,z_j)\leq 2d(y_j,z_i),
  \end{align*}
  so that we can use the terms in the numerator to cancel all terms in the denominator
  save for the ones corresponding to closest neighbours 
  (if $k\neq N/2$ some terms in the numerator are left over, and we bound them with constants):
  \begin{equation*}
  	\prod_{i<j} d(x_i,x_j)^{\frac{\beta \sigma_i \sigma_j}{\pi \gamma N}}
  	\leq C \pa{\prod_{1\leq i\leq k\wedge n-k} d(y_i,z_i)}^{-\frac{\beta}{\pi \gamma N}},
  \end{equation*}
  where $C$ is again a constant depending on all parameters. As soon as $N>\frac{\beta}{2\pi\gamma}$,
  factors of the latter product are integrable, thus concluding the proof.  
\end{proof}

\begin{defi}
	The random measure $\mu_{\beta,\gamma}^N$ is the law of
	\begin{equation*}
	\omega_{\beta,\gamma}^N=\sum_{i=1}^N \xi_i (\delta_{x_i}-1),
	\end{equation*}
	as a random variable taking values in $\M(\T^2)$, where positions $x_1,\dots x_n$ are sampled under $\nu_{\beta,\gamma,N}$,
	whenever the latter is well-defined. 
\end{defi}
In dealing with limits as $N$ goes to infinity, Gibbs measure will always be
(ultimately) defined, so we will ignore the issue henceforth in this section. Finally, let us note that
$\omega_{\beta,\gamma}^N$ can be regarded as random variables in $\dot H^s(\T^2)$ for all $s<-1$,
since signed measures have uniformly bounded Fourier coefficients.

\subsection{Energy-Enstrophy Gaussian Measures}\label{ssec:energyenstrophy}

For $\gamma>0$ and $\beta\geq 0$, let $\omega_{\beta,\gamma}$ be the centred, zero averaged,
Gaussian random field on $\T^2$ with covariance
\begin{equation*}
	\forall f,g\in \dot L^2(\T^2), \quad 	
	\expt{\brak{\omega_{\beta,\gamma},f}\brak{\omega_{\beta,\gamma},g}}=\brak{f,Q_{\beta,\gamma}g}, 
	\quad Q_{\beta,\gamma}=(\gamma-\beta\Delta)^{-1}.
\end{equation*}
Equivalently, $\omega_{\beta,\gamma}$ is a centred Gaussian stochastic process indexed by $\dot L^2(\T^2)$ with the 
specified covariance. Since the embedding of $Q_{\beta,\gamma}^{1/2}\dot L^2(\T^2)$ into $\dot H^s(\T^2)$ is Hilbert-Schmidt for all $s<-1$,
$\omega_{\beta,\gamma}$ can be identified with a random distribution taking values in the latter spaces (see \cite{dpz}). 
The special case $\beta=0$ ($\gamma=0$ will not be included in our discussion) is the \emph{white noise} on $\T^2$. 
We will denote by $\mu_{\beta,\gamma}$ the law
of $\omega_{\beta,\gamma}$ on $\dot H^s(\T^2)$, any $s<-1$. This measure is the one we formally defined in (\ref{energyenstrophymeasure}):
we will provide a rigorous interpretation of that expression in this paragraph. 
The Gaussian random distributions we just introduced are best understood in terms of Fourier series:
we can write
\begin{equation*}
	\omega_{\beta,\gamma}=\sum_{k\in\Z^2_0} \hat\omega_{\beta,\gamma,k} e_k, 
	\quad \text{ where }
	\hat\omega_{\beta,\gamma,k}=\brak{\omega_{\beta,\gamma},e_k}\sim N_\C\pa{0,\frac{4\pi^2 |k|^2}{\beta+4\pi^2 |k|^2\gamma}}
\end{equation*}
are independent $\C$-valued Gaussian variables, and the Fourier expansion thus converges in $L^2\pa{H^s(\T^2),\mu_{\beta,\gamma}}$
for $s<-1$.
The measure $\mu_{\beta,\gamma}$ is also characterised by its Fourier transform (characteristic function) on $\dot H^s(\T^2)$:
for any $f\in \dot H^{-s}(\T^2)$,
\begin{equation}\label{gaussianchartorus}
 \int e^{\imm\brak{\omega,f}} d\mu_{\beta,\gamma}(\omega)=\exp\pa{-\frac{1}{2}\sum_{k\in\Z^2_0} 
 	\frac{4\pi^2 |k|^2 |\hat f_k|^2}{\beta+4\pi^2 |k|^2\gamma}}.
\end{equation}

The main result of this section is convergence of Gibbs ensemble of vortices $\mu_{\beta,\gamma}^N$
defined above to the energy-enstrophy measure $\mu_{\beta,\gamma}$.
Let us first provide some further insight on the analogy between those random measures,
first pointed out by Kraichnan (\cite{kraichnan}). We begin by recalling an equivalent definition of $\mu_{\beta,\gamma}$:
for a smooth vorticity distribution $\omega$, energy is given by
\begin{equation*}
2E(\omega)=-\brak{\omega,\Delta^{-1}\omega}=\sum_{k\in\Z^2_0} \frac{|\hat \omega_k|^2}{4\pi^2|k^2|},
\end{equation*}
which does not make sense as a random variable if instead $\omega$ has white noise law $\mu_{0,\gamma}=\mu_\gamma$,
since in that case $\hat\omega_k$'s are i.i.d. Gaussian variables, and the series diverges almost surely.
However, one can define a \emph{renormalised energy} by means of normal ordering:
\begin{equation}\label{normenergy}
2\wick{E}=\lim_{K\rightarrow\infty}  \sum_{|k|\leq K} \frac{\wick{\hat \omega_k\hat\omega_k^\ast}}{4\pi^2|k^2|}=
\lim_{K\rightarrow\infty} \sum_{|k|\leq K} \pa{\frac{|\hat \omega_k|^2}{4\pi^2|k^2|} 
	-\int \frac{|\hat \omega_k|^2}{4\pi^2|k^2|} d\mu_\gamma(\omega)} 
\end{equation}
where the limit holds in $L^2(\mu_{\gamma})$ (see \cite{ardfhk} and \autoref{thm:clttorus} below), and it defines an element of
the second Wiener chaos $H^{\wick{2}}(\mu_{\gamma})$. As a consequence, $\wick{E}$ can be expressed
as a double It\={o}-Wiener stochastic integral with respect to the white noise $\mu_{\gamma}$,
the kernel being naturally Green's function $G$:
\begin{equation*}
	2\wick{E}(\omega)=\int_{\T^{2\times 2}} G(x,y) \wick{d\omega(x)d\omega(y)}.
\end{equation*}

\begin{lem}\label{lem:energyenstrophy}
 The probability measure on $\dot H^s(\T^2)$, any $s<-1$, defined by density as
 \begin{equation}\label{normenergymeasure}
 d\tilde\mu_{\beta,\gamma}=\frac{1}{Z_{\beta,\gamma}}e^{-\beta\wick{E}(\omega)}d\mu_\gamma(\omega),
 \quad Z_{\beta,\gamma}=\int e^{-\beta\wick{E}(\omega)}d\mu_\gamma(\omega),
 \end{equation}
 is well-posed. It coincides with the energy-enstrophy measure, $\tilde\mu_{\beta,\gamma}=\mu_{\beta,\gamma}$.
\end{lem}

The computations we perform in the forthcoming proof find analogues in the infinite product representations
of energy-enstrophy measures given for instance in \cite{albeveriocruzeiro,bpp}.

\begin{proof}
 The variable $\wick{E}$ has exponential moments because it belongs to the second Wiener chaos, so the partition
 function is finite and the measure well-defined. If characteristic functionals $\expt{e^{\imm \brak{f,\omega}}}$
 coincide for all $f\in \dot H^{-s}(\T^2)$, the two measures coincide. Since under $\mu_{0,\gamma}$ the Fourier modes
 $\hat \omega_k$ are independent centred $\C$-valued Gaussian variables with variance $\gamma^{-1}$,
 we can compute
 \begin{align*}
 	\int e^{\imm\brak{f,\omega}-\beta\wick{E}(\omega)}d\mu_\gamma
 	&= \int \exp\pa{\sum_{k\in\Z^2_0} \imm \hat f_k \hat \omega_k^\ast
 		-\beta\frac{|\hat\omega_k|^2-\gamma^{-1}}{8\pi^2 |k|^2}} d\mu_\gamma\\
 	&=\prod_{k\in\Z^2_0} \int_\C \frac{\gamma}{2\pi}\exp\pa{\imm \hat f_k z^\ast 
 		-\beta\frac{|z|^2-\gamma^{-1}}{8\pi^2 |k|^2}-\frac{\gamma|z|^2}{2}} dz\\
 	&=\prod_{k\in\Z^2_0} \frac{4\pi^2 \gamma |k|^2}{\beta+4\pi^2 \gamma |k^2|}e^{\frac{\beta}{8\pi^2 |k^2|}}
 	\exp\pa{-\frac{|\hat f_k|^2}{2} \cdot \frac{4\pi^2|k|^2 }{4\pi^2\gamma |k^2|+\beta}},
 \end{align*}
 and since the partition function $Z_{\beta,\gamma}$ can be evaluated setting $f\equiv 0$ in the latter formula,
 \begin{equation*}
 	Z_{\beta,\gamma}^{-1}\int e^{\imm\brak{f,\omega}-\beta\wick{E}(\omega)}d\mu_\gamma=
 	\prod_{k\in\Z^2_0}\exp\pa{-\frac{|\hat f_k|^2}{2} \cdot \frac{4\pi^2|k|^2 }{4\pi^2\gamma |k^2|+\beta}},
 \end{equation*}
 where the right-hand side is the characteristic function of $\mu_{\beta,\gamma}$, (\ref{gaussianchartorus}).
\end{proof}

Looking back at point vortices, the Hamiltonian function $H$ can be seen as a renormalised energy to the extent that
it includes all mutual interactions save the ones of vortices with themselves. 
To make this intuition more precise, let us first recall that in the Gaussian case $\omega\sim\mu_{0,1}$ (white noise),
the double It\=o-Wiener integral of a smooth function $h\in C^\infty(\T^{2\times 2})$ is given by
\begin{equation}\label{doubleintegralsmooth}
	\int_{\T^{2\times 2}}h(x,y)\wick{d\omega(x)d\omega(y)}=\int_{\T^{2\times 2}}h(x,y)d\omega(x)d\omega(y)-\int_{\T^2}h(x,x)dx,
\end{equation}
where integration against $d\omega(x)d\omega(y)$ is understood as the (almost surely defined) integral against
the tensor product of the random distribution $\omega$ with itself (see \cite[Chapter 7]{janson}, which includes a
discussion on how Wick ordering in double stochastic integrals can be seen as removing singular self-interactions, \emph{cf.} Remark 7.27).
By continuity on $L^2(\T^{2\times 2})$ of the double It\=o integral, the renormalised energy can be expressed as
\begin{equation}
\label{energyitotorus}
2\wick{E}(\omega)
=\lim_{n\rightarrow\infty} \int_{\T^{2\times 2}} G_n(x,y)d\omega(x)d\omega(y),\\
\end{equation}
where $G_n\in C^\infty(\T^{2\times 2})$ are symmetric and vanish on the diagonal, $G_n$ converge to $G$ in $L^2(\T^{2\times 2})$,
and the limit holds in $L^2(\mu_{\gamma})$.

In the case of a point vortices cluster $\omega^N\sim \mu_{0,\gamma}^N$, one can define renormalised double integrals in an analogous way.
Considering centred distributions (as it is $\mu_{0,1}$) is essential in the forthcoming Lemma, 
and in the case of point vortices on $\T^2$ the condition is ensured if we consider the zero average setting.

\begin{lem}\label{lem:vorticesisometry}
 Let $\omega^N\sim \mu_{0,\gamma}^N$.
 On continuous functions $h\in C(\T^{2\times 2})$ with zero average in both variables and vanishing on the diagonal, 
 \emph{i.e.} $h(x,x)=0$ for all $x$, define the map
 \begin{equation*}
 	h\mapsto \int_{\T^{2\times 2}}h(x,y)d\omega^N(x)d\omega^N(y)=\sum_{i\neq j}\xi_i\xi_j h(x_i,x_j).
 \end{equation*}
 Since it holds
 \begin{equation*}
 	\expt{\pa{\sum_{i\neq j}\xi_i\xi_j h(x_i,x_j)}^2}\leq C_\gamma \norm{h}^2_{L^2(\T^{2\times 2})}
 \end{equation*}
 with $C_\gamma$ a constant independent of $N$, the map takes values in $L^2(\mu_{0,1}^N)$, and it extends by density to
 a bounded linear map which we will denote
 \begin{equation*}
 	\dot L^2(\T^{2\times 2}) \ni f\mapsto \int_{\T^{2\times 2}}f(x,y)\wick{d\omega^N(x)d\omega^N(y)} \in L^2(\mu_{0,1}^N).
 \end{equation*}
\end{lem}

\begin{proof}
	For any function $h$ as above it holds
	\begin{align*}
	&\expt{\pa{2\sum_{i<j}^N \xi_i\xi_j h(x_i,x_j)}^2}
	= 4 \sum_{i<j}^N\sum_{\ell<k}^N \xi_i\xi_j \xi_\ell \xi_k \expt{h(x_i,x_j)h(x_\ell,x_k)}\\
	&\quad =\frac{4}{ \gamma^2 N^2}\sum_{i<j}^N \int_{\T^{2\times 2}} h(x,y)^2dxdy
	=2\frac{N-1}{\gamma^2 N}\int_{\T^{2\times 2}} h(x,y)^2dxdy,
	\end{align*}
	where the middle passage makes essential use of the zero average condition: all summands except the ones with $i=\ell,j=k$ vanish.
\end{proof}
This construction is analogous to the one of double stochastic integrals with respect to Gaussian measures (It\=o-Wiener integrals)
and Poisson point process; the above computation is also an important tool in \cite{flandoli}.
Define, in analogy with \eqref{energyitotorus}, the renormalised energy in the vortices ensemble $\mu_\gamma^N$ case as the 
renormalised double integral of the potential $G$ with respect to
$\mu_{\gamma}^N$, that is as a random variable in $L^2(\mu_{\gamma}^N)$:
by \autoref{lem:vorticesisometry}, considering approximations $G_n$ of $G$ as above,
we actually recover the Hamiltonian:
\begin{equation*}
		2\wick{E}(\omega^N)=\int_{\T^{2\times 2}} G(x,y)\wick{d\omega^N(x)d\omega^N(y)}=\sum_{i\neq j} \xi_i \xi_j G(x_i,x_j)=2H(x_1,\dots x_n).
\end{equation*}
The convergence of Hamiltonian functions of point vortices to the renormalised Gaussian energy in the case $\beta=0$
is an important part in the proof of the forthcoming main result of this Section.
\begin{thm}\label{thm:clttorus}
	Let $\beta/\gamma\geq 0$. It holds:
	\begin{enumerate}
		\item[(1)] $\lim_{N\rightarrow\infty} Z_{\beta,\gamma,N}=Z_{\beta,\gamma}$;
		\item[(2)] the sequence of $\M$-valued random variables $\omega^N\sim\mu_{\beta,\gamma}^N$ 
		converges in law on $\dot H^s(\T^2)$, any $s<-1$, to a random distribution $\omega\sim \mu_{\beta,\gamma}$,
		as $N\rightarrow\infty$;
		\item[(3)] the sequence of real random variables $H(\omega^N)$ converges in law to $\wick{E}(\omega)$ as 
		$N\rightarrow\infty$, with $\omega^N,\omega$ as in point (2).
	\end{enumerate}
\end{thm}

In \autoref{prop:zgreentorus} we needed to impose that $\beta/\gamma$ be small in order for the Gibbs measure to exist.
Let us remark once again that the constraint depended on $N$, and was always satisfied for large enough $N$.

\subsection{Potential Splitting and the Sine-Gordon transformation}\label{ssec:sinegordon}

In this paragraph we introduce the key tools in the proof of \autoref{thm:clttorus}.
The main issue is the logarithmic singularity of the Green function $G$. To deal with it we will decompose $G$
in two parts, a smooth approximation of $G$ and a remainder retaining logarithmic singularity: for $m>0$,
\begin{equation}\label{potentialsplitting}
	G=-\Delta^{-1}=\pa{-\Delta^{-1}-(m^2-\Delta)^{-1}}+(m^2-\Delta)^{-1}:= V_m+W_m.
\end{equation}
Physically, the smooth part $V_m$ corresponds to the long-range part of the potential, and the singular part $W_m$ to 
short-range interactions. We will also denote
\begin{equation*}
	H=H_{V_m}+H_{W_m}=\sum_{i< j}^N \xi_i\xi_j V_m(x_i,x_j) + \sum_{i< j}^N \xi_i\xi_j W_m(x_i,x_j),
\end{equation*}
the relative splitting of the Hamiltonian. In terms of Fourier series,
\begin{equation*}
   W_m(x,y)=\sum_{k\in\Z^2_0} \frac{e_k(x-y)}{m^2+4\pi^2|k|^2}, \quad
   V_m(x,y)=\sum_{k\in\Z^2_0} \frac{m^2 e_k(x-y)}{4\pi^2|k|^2(m^2+4\pi^2|k|^2)}.
\end{equation*}
The Green function $W_m$ is called the 2-dimensional \emph{Yukawa potential} 
or \emph{screened Coulomb potential} with mass $m$
(as opposed to the \emph{Coulomb potential} $G$). 

We will regard the regular part of the Hamiltonian corresponding to $V_m$ as the covariance
of a Gaussian field. The idea, dating back to \cite{samuel}, originated as a connection between 
the classical Coulomb gas theory and sine-Gordon field theory (hence the name):
it will allow us to analyse the convergences in \autoref{thm:clttorus} by standard
Gaussian computations, up to a remainder term involving the Yukawa potential $W_m$
(whose associated partition function we bound in \autoref{ssec:partitionfunctions}).
We thus define $F_m$ as the centred Gaussian field on $\T^2$ with covariance kernel $V_m$, that is
\begin{equation}
\forall f,g\in \dot L^2(\T^2), \quad 	
\expt{\brak{F_m,f}\brak{F_m,g}}=\brak{f,\pa{-\Delta^{-1}-(m^2-\Delta)^{-1}} g}.
\end{equation}
The remainder of this paragraph deals with properties of $F_m$. The reproducing kernel Hilbert space is
\begin{equation*}
	\sqrt{-\Delta^{-1}-(m^2-\Delta)^{-1}}\dot L^2(\T^2)\subseteq \dot H^2(\T^2),
\end{equation*} 
so that $F_m$ has a $\dot H^s(\T^2)$-valued version for all $s<1$, into which $\dot H^2(\T^2)$ has Hilbert-Schmidt embedding.
As a consequence, by Sobolev embedding, $F_m$ has a version taking values in $\dot L^p(\T^2)$ for all $p\geq 1$.

The field $F_m$ can also be evaluated at points $x\in\T^2$: the coupling $F_m(x):=\brak{\delta_x,F_m}$ is
defined as the series, converging in $L^2(F_m)$ \emph{uniformly in} $x\in\T^2$,
\begin{equation*}
	\brak{\delta_x,F_m}=\sum_{k\in\Z^2_0} e^{2\pi \imm x\cdot k} \hat F_{m,k},
	\quad \hat F_{m,k}=\brak{e_k,F_m}\sim N_\C\pa{0,\frac{m^2}{4\pi^2 |k|^2\pa{m^2+4\pi^2|k|^2}}}.
\end{equation*}
In other terms, $x\mapsto F_m(x)$ is a measurable random field, and 
$F_m(x)$ are centred Gaussian variables of variance $V_m(x,x)=V_m(0,0)$.
A straightforward application of Kolmogorov continuity theorem shows that there exists a version of $F_m(x)$
which is $\alpha$-H\"older for all $\alpha<1/2$.

\begin{lem}\label{lem:Fboundstorus}
	For any $\alpha>0$, $p\geq 1$ and $m\rightarrow \infty$,
	\begin{align}
	\label{Fmomentstorus}
	\expt{\norm{F_m}_p^p}&\simeq_p (\log m)^{p/2}\\
	\label{Fexpmomentstorus}
	\expt{\exp\pa{-\alpha\norm{F_m}_2^2}}&\simeq m^{-\frac{\alpha}{2\pi}}.
	\end{align}
\end{lem}

\begin{proof}
 Let us begin with moments: by Fubini-Tonelli theorem,
 \begin{equation*}
 	\expt{\norm{F_m}_p^p}=\int_{\T^{2}}\expt{|F_m(x)|^p}dx=c_p \int_{\T^{2}}V_m(x,x)^{p/2}dx=c_p V_m(0,0)^{p/2},
 \end{equation*}
 where $V_m(0,0)=\frac{1}{2\pi}\log m+o(\log m)$ can be checked by explicit computation in Fourier series.
 As for exponential moments, a standard Gaussian computation (see \cite[Proposizion 2.17]{dpz}) gives
 \begin{align*}
 \expt{\exp\pa{-\alpha \norm{F_m}_2^2}}
 &=\exp \set{-\frac{1}{2}\trace\pa{\log\pa{1+2\alpha \pa{-\Delta^{-1}-(m^2-\Delta)^{-1}}}}}\\
 &=\exp\pa{-\frac{1}{2}\sum_{k\in\Z^2_0}\log \pa{1+\frac{2\alpha m^2}{4\pi^2 |k|^2 (m^2+4\pi^2|k|^2)}}}\\
 &>\exp \pa{-\sum_{k\in\Z^2_0}\frac{\alpha m^2}{4\pi^2 |k|^2 (m^2+4\pi^2|k|^2)}}\\
 &=\exp\pa{-\alpha V_m(0,0)}\simeq m^{-\frac{\alpha}{2\pi}},  
 \end{align*}
 the other inequality descending from analogous computations using $\log(1+x)>x-\frac{x^2}{2}$, $x>0$, instead of
 the inequality $\log(1+x)<x$ we just applied. 
\end{proof}

Since it holds, for $s,t\in\R$,
\begin{equation*}
	\expt{e^{\imm s F_m(x)} e^{\imm t F_m(y)}}=e^{-\frac{s^2+t^2}{2}V_m(0,0)}e^{-stV_m(x,y)},
\end{equation*}
(and analogous expressions for $n$-fold products) we can transform the partition function relative to the regular part of the Hamiltonian $H_{V_m}$:
\begin{align}\label{sinegordontorus}
  \int_{\T^{2N}} & e^{-\beta H_{V_m}} dx_1\cdots dx_n\\ \nonumber
  &= \int_{\T^{2N}} \exp\pa{-\beta\sum_{i\neq j}^{N} \frac{\sigma_i\sigma_j}{2\gamma N} V_m(x_i,x_j)}dx_1\cdots dx_n\\ \nonumber
   &=e^{\frac{\beta}{2\gamma}V_m(0,0)}
   \expt{\int_{\T^{2N}} \exp\pa{-\imm \sqrt{\frac{\beta}{\gamma N}}\sum_{i=1}^{N} \sigma_i F_m(x_i)}dx_1\cdots dx_n}.
\end{align}
Rewriting the partition function in these terms is the first step in the analysis of $Z_{\beta,\gamma,N}$,
the next one being a control of the singular part of the potential, which we could not transform.
We deal with $W_m$ in the next paragraph: let us conclude the present one with the estimate we will use
on complex exponentials of $F_m$. It relies essentially on:

\begin{lem}\label{lem:exponentialintegral}
	If $f\in \dot L^4(\T^2)$, then
	\begin{equation*}
	\abs{\int_{\T^2} e^{\imm f(x)} dx-e^{-\frac{1}{2}\norm{f}^2_{2}}}\leq \frac{\norm{f}^3_3}{6}+\frac{\norm{f}^4_2}{8}.
	\end{equation*} 
\end{lem}

\begin{proof}
	Thanks to the zero average condition, we can expand
	\begin{align*}
	&\int_{\T^2} e^{\imm f(x)} dx-e^{-\frac{1}{2}\norm{f}^2_{2}}\\
	& \quad = \int_{\T^2} \pa{e^{i f(x)}-1-\imm f(x)+\frac{f(x)^2}{2}} dx
	-\pa{e^{-\frac{1}{2}\norm{f}^2_{2}}-1+\frac{\norm{f}^2_2}{2}}
	\end{align*}
	and then apply Taylor expansions
	\begin{equation*}
	\abs{e^{it}-1-it+\frac{t^2}{2}}\leq \frac{t^3}{6}, \quad \abs{e^{-t}-1+t}\leq\frac{t^2}{2}.
	\end{equation*}
\end{proof}

\begin{prop}\label{prop:Fcomplexboundstorus}
  For any $\beta,\gamma>0$ and integer $p\geq 1$, if $m=m(N)$ grows at most polynomially in $N$, then it holds
  \begin{equation*}
  	\int_{\T^{2N}} e^{-\beta H_{V_m}} dx_1\cdots dx_n
  	\leq C_{\beta,\gamma,p} \pa{1+\frac{m^{\frac{\beta}{4\pi\gamma}}\pa{\log m}^{2p}}{N^{p/2}}}
  \end{equation*}
  uniformly in $N$.	
\end{prop}

To ease notation, in the following argument we will denote
\begin{equation*}
E_j=\int_{\mathbb{T}_2}e^{\imm\xi_j\sqrt{\beta}F_m(x_j)}\,dx_j,
\qquad \Es =e^{-\frac{\beta}{2N\gamma}\|F_m\|_{L^2}^2},
\end{equation*}
(notice that both depend on $N,m$) and thus write \eqref{sinegordontorus} as
\begin{equation*}
\int_{\T^{2N}} e^{-\beta H_{V_m}} dx_1\cdots dx_n =e^{\frac{\beta}{2\gamma}V_m(0,0)}
\expt{\prod_{j=1}^N E_j}
\end{equation*}	
In sight of \autoref{lem:exponentialintegral}, we expect the $0$-th order term (in $1/N$)
to be $e^{\frac{\beta}{2\gamma}V_m(0,0)}\expt{\Es^N}$, which is $O(1)$ as shown above in \autoref{lem:Fboundstorus}.
The forthcoming proof applies the Taylor expansion of \autoref{lem:exponentialintegral} to further
and further orders.

\begin{proof}
 For $p=1$, we expand the product $\prod_{j=1}^N E_j$ by means of the algebraic identity 
 \begin{equation}\label{algfirstorder}
 	\prod_{j=1}^N E_j=\Es^N+\sum_{k=1}^{N}(E_k-\Es)\Es^{N-k}\pa{\prod_{j=1}^{k-1} E_j},
 \end{equation}
 from which we can estimate
 \begin{align*}
 &\expt{\prod_{j=1}^N E_j}
 =\expt{\Es^N}+ \sum_{k=1}^{N} \expt{\pa{\prod_{j=1}^{k-1} E_j}(E_k-\Es)\Es^{N-k}}\\
 &\qquad\leq \expt{\Es^N}+ \sum_{k=1}^{N} \expt{\abs{E_k-\Es}}\\
 &\qquad\leq \expt{\Es^N}
 +N\cdot \expt{\frac{1}{6}\pa{\frac{\beta}{\gamma N}}^{3/2}\norm{F_m}_3^3+\frac{1}{8}\pa{\frac{\beta}{\gamma N}}^2 \norm{F_m}_2^4}\\
 &\qquad \leq C_{\beta,\gamma} \pa{m^{-\frac{\beta}{4\pi\gamma}} +\frac{(\log m)^2}{\sqrt{N}}}.
 \end{align*}
 The higher order terms (in $1/N$) have been dealt with in the following way:
 exponential factors have been bounded with $|E_j|,|\Es|\leq 1$, only leaving differences $E_k-\Es$ from which smallness is obtained.
 The third step is the crucial application of \autoref{lem:exponentialintegral},
 and the last one is \autoref{lem:Fboundstorus} and H\"older inequality. The thesis now follows recalling once again that
 $V_m(0,0)\simeq \frac{1}{2\pi}\log m$.
 
 For $p=2$, by iterating \eqref{algfirstorder} we get the identity
 \begin{equation*}
 \prod_{j=1}^N E_j = \Es^N + \Es^{N-1}\sum_{k=1}^N(E_k-\Es)
 + \sum_{k=2}^N\sum_{\ell=1}^{k-1}(E_\ell-\Es)(E_k-\Es)\Es^{N-\ell-1}
 {\prod_{j=1}^{\ell-1}E_j}.
 \end{equation*}
 Taking expectations and controlling separately the summands,
 \begin{align*}
 \expt{\prod_{j=1}^N E_j}
 &\leq \expt{\Es^N}+ \sum_{k=1}^{N} \expt{|E_k-\Es|\Es^{N-k}}
 +\sum_{k=2}^N\sum_{\ell=1}^{k-1} \expt{\abs{E_\ell-\Es}\abs{E_k-\Es}}\\
 &\leq \expt{\Es^N}+N \expt{\Es^{2(N-k)}}^{1/2} \expt{|E_1-\Es|^2}^{1/2}\\
 &\quad +\frac12 N(N-1)\expt{|E_1-\Es|^2}\\
 &\lesssim m^{-\frac{\beta}{4\pi\gamma}} + m^{-\frac{\beta(N-k)}{4\pi\gamma N}} N^{-1/2}(\log m)^{3/2}+N^{-1}(\log m)^3.
 \end{align*}
 In the latter computation, the second step is Cauchy-Schwarz inequality, while the third combines 
 H\"older inequality and \autoref{lem:Fboundstorus} to control
 \begin{align*}
 	\expt{|E_k-\Es|^2}\lesssim \expt{\pa{N^{-3/2} \norm{F_m}_3^3 + N^{-2} \norm{F_m}_2^4}^2}\lesssim N^{-3}(\log m)^3.
 \end{align*}
 The thesis for $p=2$ is obtained, since we have shown that
 \begin{equation*}
 	\int_{\T^{2N}} e^{-\beta H_{V_m}} dx_1\cdots dx_n\lesssim 1+ N^{-1/2}(\log m)^{3/2}+ m^{\frac{\beta}{4\pi\gamma}} N^{-1}(\log m)^3,
 \end{equation*}
 where the middle term is always $o(1)$ because we are assuming that $m(N)$ grows at most polynomially.
 
 Further iterations of \eqref{algfirstorder} to expand products of $E_j$ produce 
 in a completely analogous manner the required estimate for arbitrary $p\geq 1$. 
 Let us only report, as an example, the third order iteration of \eqref{algfirstorder}:
 \begin{align*}
 	\prod_{j=1}^N E_j
 	&= \Es^N
 	+ \Es^{N-1}\sum_{k=1}^N(E_k-\Es)
 	+ \Es^{N-2}\sum_{k=2}^N\sum_{\ell=1}^{k-1}(E_\ell-\Es)(E_k-\Es)\\
 	&\quad + \sum_{k=3}^N\sum_{\ell=2}^{k-1}\sum_{m=1}^{\ell-1}
 	(E_k-\Es)(E_\ell-\Es)(E_m-\Es)\Es^{N-m-2}
 	\Bigl(\prod_{j=1}^{m-1}E_j\Bigr). \qedhere
 \end{align*}
\end{proof}

\subsection{Controlling Partition Functions}\label{ssec:partitionfunctions}

We want to analyse separately the contributions of regular and singular parts of the potential to the partition function
\begin{equation*}
	Z_{\beta,\gamma,N}=\int_{\T^{2N}} e^{-\beta H_{V_m}} e^{-\beta H_{W_m}}dx^N.
\end{equation*}
The core idea is that if we send $m(N)\rightarrow\infty$ along $N\rightarrow\infty$ with a suitable rate, 
the contribution of the Yukawa part of the potential, $W_m$, becomes irrelevant, and we can bound $Z_{\beta,\gamma,N}$
uniformly in $N$. With a uniform bound at hand, identifying the limit becomes quite simple: we will do so 
in the next Section, reducing ourselves to the case $\beta=0$.

Let us thus focus on $W_m$. Its free version $W_{m,\R^2}$, that is the Green function
of $m^2-\Delta$ on the whole plane, can be expressed in term of the modified Bessel function of the second kind $K_0$ as
\begin{equation}\label{yukawagreenplane}
W_{m,\R^2}(x,y)=W_{m,\R^2}(|x-y|)=\frac{1}{2\pi}K_0(m|x-y|), \quad x,y\in\R^2,
\end{equation}
where $K_0$ is the positive solution of
\begin{equation*}
r^2 K_0''(r)+r K_0'(r)-r^2 K_0(r)=0, \quad r\geq 0,
\end{equation*}
with logarithmic divergence in $r=0$ and exponential decay for large $r$,
\begin{align}
\label{klogsingular}
K_0(r)&=-\log(r)+O(1), &r\rightarrow 0,\\
\label{kexpdecay}
K_0(r)&\leq \frac{\sqrt\pi e^{-r}}{\sqrt 2 r}, &\forall r>0
\end{align}
(see \cite{abramowitz}). Unlike $G_{\R^2}(x)=-\frac{1}{2\pi}\log|x|$, $W_{m,\R^2}\in L^1(\R^2)$, 
hence by Poisson summation formula it holds, for any  distinct $x,y\in\T^2$,
\begin{equation}\label{yukawaquotient}
W_m(x,y)=\sum_{k\in\Z^2} W_{m,\R^2}(|x+k-y|)-\int_{\R^2} W_{m,\R^2}(|x|)dx,
\end{equation}
the integral on right-hand side taking care of the space average.
Notice that, since $K_0$ is positive, so is the first summand in \eqref{yukawaquotient}.
This representation allows for a quite precise control of $W_m$, which we now use
to control the rate at which the partition function relative to Yukawa potential goes to 1 as $m\rightarrow\infty$.

\begin{prop}\label{prop:zyukawatorus}
	Let $N\geq 1$, $\beta/\gamma> -8\pi$ and $m>0$. 
	There exists a constant $C_{\beta,\gamma}>0$ such that
	\begin{equation*}
	\int_{\T^{2N}} e^{-\beta H_{W_m}}dx_1\cdots dx_n \leq \pa{1+C_{\beta,\gamma}\frac{(\log m)^2}{m^2}}^N
	\end{equation*}
	(uniformly with respect to the choice of signs $\sigma_i$).
\end{prop}

\begin{proof}
	As a first step we produce an estimate on $W_m(x)=W_m(x,0)$ which separates the short-range, 
	relevant part and a long range remainder.
	We do so by means of the representation (\ref{yukawaquotient}), so first we have to take a closer
	look at $W_{m,\R^2}$.
	We choose a small radius $\frac{1}{m}\ll r_m=\frac{2\log m}{m} \ll 1$, below which we control $W_{m,\R^2}$ with logarithm:
	by (\ref{kexpdecay}), and since $K_0$ is decreasing, $W_{m,\R^2}(x)\leq \frac{C}{m^2}$ when 
	$|x|\geq r_m$ ($C$ will denote possibly different positive constants throughout this proof). 
	Inside the ball $B(0,r_m)$, by comparison principle,
	\begin{equation}\label{comparisonprinciple}
	\forall x\in \bar B(0,r_m) \quad W_{m,\R^2}(x)\leq -\frac{1}{2\pi}\log\pa{\frac{|x|}{r_m}}+\frac{C}{m^2}
	\end{equation}
	since the right-hand side is the solution to the problem
	\begin{equation*}
	\begin{cases}
	-\Delta u=\delta_0 &\text{in }B(0,r_m)\\
	u=\frac{C}{m^2} &\text{in }\partial B(0,r_m)\\
	\end{cases}.
	\end{equation*}
	Applying \eqref{kexpdecay} we can bound
	\begin{equation*}
		\sum_{k\in\Z^2_0} W_{m,\R^2}(|x+k|)\leq C \sum_{k\in\Z^2_0} e^{-m|k|}\leq \frac{C}{m^2},
	\end{equation*} 
	so going back to \eqref{yukawaquotient}, we control separately the summand $k=0$ with \eqref{comparisonprinciple}
	and the others as above, to get
	\begin{equation}\label{Wexpansiontorus}
	0<\sum_{k\in\Z^2} W_{m,\R^2}(|x+k|)=\leq -\frac{1}{2\pi}\log\pa{\frac{d(x,0)}{r_m}}\chi_{B(0,r_m)}(x)+\frac{C}{m^2}
	\end{equation}
	(compare with the expansion \eqref{greentorus}).
	Change of variables and \eqref{yukawagreenplane} show that also
	\begin{equation}\label{Waverage}
	0<\int_{\R^2} W_{m,\R^2}(|x|)dx \leq \frac{C}{m^2}.
	\end{equation}
	
	We now apply H\"older's inequality to obtain the thesis in the regime $|\beta/\gamma|<8\pi$.
	Keeping in mind that $W_m$ is translation invariant,
	\begin{align*}
	\int_{\T^{2N}} e^{-\beta H_{W_m}}dx_1\cdots dx_n
	&=\int_{\T^{2N}} \prod_{i=1}^{N} \prod_{j\neq i, j=1}^{N} \exp\pa{-\frac{\beta \sigma_i\sigma_j}{2\gamma N} W_m(x_i,x_j)}
	dx_1\cdots dx_n\\
	&\leq \prod_{i=1}^{N} \pa{ \prod_{j\neq i, j=1}^{N} \int_{\T^{2}}
		\exp\pa{-\frac{\beta \sigma_i\sigma_j}{2\gamma} W_m(x_j,0)} dx_j}^{1/N},
	\end{align*}
	so we can restrict ourselves to the case of two particles. Since we are already neglecting possible cancellations
	due to signs (and allowing for negative inverse temperatures $\beta$), they are irrelevant:
	let us say they are opposite to fix ideas.
	Applying the above pointwise estimates then leads to  
	\begin{align}\label{Wsingleintegral}
	\int_{\T^{2}}\exp\pa{\frac{\beta}{2\gamma} W_m(x)} dx
	&\leq \pa{1+\int_{d(x,0)\leq r_m}
		\pa{\frac{d(x,0)}{r_m}}^{\frac{\beta}{4\pi\gamma}}dx}e^{C/m^2}\\ \nonumber
	&\leq \pa{1+C r_m^2}e^{C/m^2}= 1+O\pa{\frac{(\log m)^2}{m^2}}
	\end{align}
	as soon as $\frac{\beta}{\gamma}<8\pi$ for integrability, from which the thesis follows.
	
	To cover all positive temperatures $\beta/\gamma\geq 0$, we resort instead to the technique employed in \autoref{prop:zgreentorus}.
	Assume first that positive and negative vortices are in equal number, and relabel them by minimal distance dipoles
	as in \autoref{prop:zgreentorus} (see \eqref{minimaldipoles}, whose notation we employ in the following).
	Then we can group the summands of the Hamiltonian function as
	\begin{align}\label{coupling}
	H_{W_m} 
	&= \frac1{\gamma N}\sum_{i<j}(W_m(y_i-y_j)-W_m(z_i,y_j))\\ \nonumber
	&\quad + \frac1{\gamma N}\sum_{i<j}(W_m(z_i-z_j)-W_m(y_i,z_j))
	- \frac1{2\gamma N}\sum_i W_m(y_i,z_i).
	\end{align}
	The first and second term in the formula above are similar, so we only look at the first one. 
	There are two possible cases to consider. For $i<j$,
	\begin{itemize}
		\item if $d(z_i,y_j)>\frac{r_m}{2}$, by \eqref{Wexpansiontorus} and \eqref{Waverage} it holds
		\begin{align*}
		W(z_i,y_j) - W_m(y_i,y_j) \leq -\frac1{2\pi}\log\pa{\frac{d(z_i,y_j)}{r_m}}+\frac{C}{m^2}\lesssim \frac1{m^2};
		\end{align*}
		\item if $d(z_i,y_j)\leq \frac{r_m}{2}$, then it must be $d(y_i,z_i)\leq \frac{r_m}{2}$, and thus
		\begin{equation*}
		d(y_i,y_j)\leq d(y_i,z_i)+d(z_i,y_j)\leq 2 d(z_i,y_j)\leq r_m,
		\end{equation*}
		so that we can bound, again by \eqref{Wexpansiontorus} and \eqref{Waverage},
		\begin{align*}
		W(z_i,y_j) - W_m(y_i,y_j)
		&\leq -\frac1{2\pi}\log\pa{\frac{d(z_i,y_j)}{r_m}}
		+ \frac1{2\pi}\log\pa{\frac{d(y_i,y_j)}{r_m}} +\frac{C}{m^2}\\
		&\leq \frac1{2\pi}\log\pa{\frac{d(y_i,y_j)}{d(z_i,y_j)}}+\frac{C}{m^2} \leq C.
		\end{align*}
	\end{itemize}
    We conclude that, in either case,
    \begin{equation*}
    	W(z_i,y_j) - W_m(y_i,y_j)\leq C\pa{\chi_{d(y_i,z_i)\leq r_m/2}+\frac1{m^2}}
    \end{equation*}
	Applying these estimates to the first and second sums in \eqref{coupling}, we can control the Gibbsian exponential density by
	\begin{equation*}
		e^{-\beta H_{W_m}}
		\leq 
		\prod_{i=1}^N e^{\frac{\beta}{2\gamma N} W_m(y_i,z_i)} 
		e^{C_{\beta,\gamma}\pa{\chi_{d(y_i,z_i)\leq r_m/2}+\frac1{m^2}}},
	\end{equation*}
	so that, integrating over all variables,
	\begin{equation*}
	\int_{\T^{2N}}e^{-\beta H_{W_m}}dx^N
	\leq\pa{\int_{\T^4} e^{\frac{\beta}{2\gamma N} W_m(y,z)}
		e^{C\pa{\chi_{d(y,z)\leq r_m/2}+\frac1{m^2}}}dydz}^N.
	\end{equation*}
	We are now able to control the two exponentials separately by Cauchy-Schwarz inequality and \eqref{Wexpansiontorus}, \eqref{Waverage}.
	If $0<\delta<\frac{4\pi \gamma}{\beta}$, the same computation of \eqref{Wsingleintegral} leads to
	\begin{equation*}
	\int_{\T^{4}} e^{\frac\beta{\gamma N}W_m(y,z)}\,dy\,dz
	\leq \pa{\int_{\T^{4}}e^{\frac{\delta \beta}{\gamma}W_m(y,z)} \,dy\,dz}^{\frac1{\delta N}}
	\leq \pa{1+C r_m^2}^{\frac1{\delta N}}e^{\frac{C}{Nm^2}},
	\end{equation*}
	while the second factor to control is
	\begin{equation*}
	\int_{\T^{2N}} e^{C\pa{\chi_{d(y,z)\leq r_m/2}+\frac1{m^2}}}dydz \leq (1+C r_m^2)  e^{\frac{C}{m^2}}.
	\end{equation*}
	The thesis now follows collecting all estimates.	
	The case in which there are more positive than negative vortices, or vice-versa, is readily settled as follows.
    Let $P_N$ and $Q_N$ be the numbers of positive and negative vortices, say $Q_N < P_N$. Then \eqref{coupling}
	becomes
	\begin{align}\label{posandneg}
	H_{W_m} &= \frac1{\gamma N}\sum_{i=1}^{Q_N}\sum_{j=i+1}^{P_N}(W_m(y_i-y_j)-W_m(z_i,y_j))\\ \nonumber
	&+ \frac1{\gamma N}\sum_{i=1}^{Q_N}\sum_{j=i+1}^{Q_N}(W_m(z_i-z_j)-W_m(y_i,z_j))
	- \frac1{2 \gamma N}\sum_{i=1}^{Q_N} W_m(y_i,z_i)\\ \nonumber
	&+ \frac1{\gamma N}\sum_{i=Q_N+1}^{P_N}\sum_{j=i+1}^{P_N}W_m(y_i-y_j).
	\end{align}
	Since it is always $W_m\gtrsim -\frac1{m^2}$, the new term appearing in \eqref{posandneg} --the fourth one in the right-hand side-- 
	contributes at most with a factor $\exp\pa{C_{\beta,\gamma}N/m^2}$ to the exponential integral, so the proof carries on as before.
\end{proof}

\begin{cor}\label{cor:zuniformbound}
	If $N\geq 1$, $\beta/\gamma\geq 0$, $Z_{\beta,\gamma,N}$ is uniformly bounded in $N$ by 
	a constant depending only on $\beta,\gamma$. 
\end{cor}

\begin{proof}
	Let $a>0$, $m(N)=N^a$ and $p\geq 1$ an integer, then by \autoref{prop:Fcomplexboundstorus} and \autoref{prop:zyukawatorus}
	we have
	\begin{align*}
		Z_{\beta,\gamma,N}&=\int_{\T^{2N}} e^{-\beta H_{V_m}} e^{-\beta H_{W_m}}dx^N\\
		&\leq \pa{\int_{\T^{2N}} e^{-2\beta H_{V_m}}dx^N}^{1/2} \pa{\int_{\T^{2N}} e^{-2\beta H_{W_m}}dx^N}^{1/2}\\
		&\leq C_{\beta,\gamma,p} \pa{1+N^{\frac{a \beta}{4\pi\gamma}-\frac{p}{2}}a^{2p}}
		\pa{1+C_{\beta,\gamma}\frac{a^2}{N^{2a}}}^N.
	\end{align*}
	The partition function $Z_{\beta,\gamma,N}$ is then uniformly bounded in $N$ as soon as
	\begin{equation*}
		\frac{\beta}{4\pi\gamma}a<\frac{p}{2},\qquad 1-2a<0.
	\end{equation*}
	Since, for any given $\beta/\gamma>0$, we can choose $p\geq 1$ in \autoref{prop:Fcomplexboundstorus} 
	large enough for the interval
	$\frac12<a<\frac{2\pi\gamma}{\beta} p$ not to be empty, the thesis follows.
\end{proof}

\begin{rmk}
	The separation of long-range relevant interaction and singular short range ones in $G=V_m+W_m$ 
	may in fact be obtained in a variety of ways: a notable mention is the decomposition $G=V_\epsilon+W_\epsilon$, 
	with
	\begin{equation*}
		V_\epsilon=e^{-\epsilon\Delta}\ast G=\int_\epsilon^\infty e^{-t\Delta}dt, \quad W_\epsilon=G-V_\epsilon=\int_0^\epsilon e^{-t\Delta}dt
	\end{equation*}
	(in fact, $V_\epsilon$ is the smoothed potential considered in \cite{bpp}). The singular part $W_\epsilon$ admits
	the representation \eqref{yukawaquotient}, with Bessel's function $K_0$ replaced by the exponential integral function $E_1$.
	The latter behaves very similarly to $K_0$: it diverges logarithmically in the origin and decays exponentially for large arguments.
	Indeed, this decomposition is completely equivalent to the one we chose for our purposes.
\end{rmk}

\begin{rmk}
	Bounds on partition functions of point vortices -or the closely related 2-dimensional Coulomb gas ensembles-
	are a central part in many works on the topic. We refer for instance to the ones obtained in \cite{deutschlavaud,gunsonpanta,bodineau}.
	However, the uniform bound we obtain with our particular scaling of intensities does not seem to be 
	obtainable from their estimates. A remarkable consequence of the results in \cite{serfaty17} is
	an asymptotic expansion in $N$ of 2-dimensional log-gas partition function, for fixed $\beta$ and charge intensities:
	even though their expansion does not provide straightforwardly 
	the uniform bound we obtain in our scaling, a careful analysis of their arguments might provide an alternative proof.
\end{rmk}

\subsection{Proof of Central Limit Theorem}\label{ssec:proofclt}

We are now able to conclude the proof of \autoref{thm:clttorus}. The first step is the case $\beta=0$,
which in fact does not rely on the above arguments, and is essentially due to \cite{flandoli}.

\begin{proof}[Proof of \autoref{thm:clttorus}, $\beta=0$.]
	The statement on partition functions is trivial in this case.
	Convergence in law of $\omega^N\sim\mu_{\gamma}^N$ to $\omega\sim \mu_{\gamma}$
	on $\dot H^s(\T^2)$, any $s<-1$, is ensured by a straightforward application of the Central Limit Theorem
	for sums of independent variables on Hilbert spaces. 
	As for the convergence of the Hamiltonian: let $G_n$ converge to $G$ in $L^2(\T^{2\times 2})$, with $G_n$ vanishing on the diagonal, 
	and split
	\begin{align*}
	\int G(x,y)& \wick{d\omega^N(x)d\omega^N(y)}-\int G(x,y) \wick{d\omega(x)d\omega(y)}\\
	&=\int G(x,y) \wick{d\omega^N(x)d\omega^N(y)}-\int G_n(x,y) d\omega^N(x)d\omega^N(y)\\
	&\quad +\int G_n(x,y) d\omega^N(x)d\omega^N(y)-\int G_n(x,y) d\omega(x)d\omega(y)\\
	&\quad +\int G_n(x,y) d\omega(x)d\omega(y)-\int G(x,y) \wick{d\omega(x)d\omega(y)}.
	\end{align*}
	The $L^2(\Omega,\PP)$-norms of the differences on the right-hand side vanish in the limit.
	Indeed, thanks to \autoref{lem:vorticesisometry}, the first one is controlled uniformly in $N$ by
	\begin{equation*}
		\expt{\abs{\int G(x,y) \wick{d\omega^N(x)d\omega^N(y)}-\int G_n(x,y) d\omega^N(x)d\omega^N(y)}^2}
		\lesssim_\gamma \norm{G-G_n}_{\dot L^2(\T^{2\times 2})}^2,
	\end{equation*}
	and the very same estimate holds for the third summand by Gaussian It\=o isometry, \emph{cf.} (\ref{energyitotorus}).
	The second moment of the middle term vanishes as $N\rightarrow\infty$ since we have already proved that 
	$\omega^N$ converges in law on $H^s(\T^2)$ for $s<-1$, so that $\omega^N\otimes \omega^N$ converges in law on $H^{2s}(\T^{2\times 2})$
	(uniform integrability descends again by the above estimate and It\=o isometry).
\end{proof}

\begin{proof}[Proof of \autoref{thm:clttorus}, $\beta>0$.]
  Consider variables $\omega_\gamma^N\sim\mu_{\gamma}^N$ converging to $\omega_\gamma\sim\mu_{\gamma}$ as above.
  We have just seen that if $\beta=0$ the Hamiltonian $H(\omega_\gamma^N)$ converges to $\wick{E}(\omega_\gamma)$
  in $L^2(\Omega,\PP)$. Since $x\mapsto e^{-\beta x}$ is a continuous function on $\R$, this implies that
  $e^{-\beta H(\omega_\gamma^N)}$ converges in probability to $e^{-\beta \wick{E}(\omega_\gamma)}$ for all $\beta\in\R$.
  If $e^{-\beta H(\omega_\gamma^N)}$ is uniformly integrable in $N$, then its expected value $Z_{\beta,\gamma,N}$
  converges to $Z_{\beta,\gamma}=\expt{e^{-\beta \wick{E}(\omega_\gamma)}}$. By \autoref{cor:zuniformbound},
  \begin{equation*}
  	\expt{\pa{e^{-\beta H(\omega_\gamma^N)}}^p}=Z_{p\beta,\gamma,N}
  \end{equation*}
  is uniformly bounded in $N$ for all $p\beta/\gamma\geq 0$. As a consequence, $e^{-\beta H(\omega_\gamma^N)}$ 
  is uniformly integrable if $\beta/\gamma\geq 0$, thus proving point (1).
  
  Since $(e^{-\beta H(\omega_\gamma^N)},\omega_\gamma^N)$ converges in law to $(e^{-\beta\wick{E}(\omega_\gamma)},\omega_\gamma)$
  on the Polish space $\R\times \dot H^s(\T^2)$, any $s<-1$, we deduce the convergence on $\dot H^s(\T^2)$ of the probability distributions
  \begin{equation*}
  	d\mu_{\beta,\gamma}^N(\omega)=e^{-\beta H(\omega)}d\mu_\gamma^N(\omega)\rightarrow
  	e^{-\beta\wick{E}(\omega)}d\mu_{\gamma}(\omega)=d\mu_{\beta,\gamma}(\omega)
  \end{equation*}
  for all $\beta\geq 0$. We are only left to prove convergence of the Hamiltonian $H(\omega_{\beta,\gamma}^N)$
  for $\omega_{\beta,\gamma}^N\sim\mu_{\beta,\gamma}^N$. Since its Laplace transform is given by
  \begin{equation*}
  	\expt{e^{\alpha H(\omega_{\beta,\gamma}^N) }}
  	=\int e^{\alpha H(\omega)}\frac{e^{-\beta H(\omega)}}{Z_{\beta,\gamma,N}}d\mu_{\gamma}^N
  	=\frac{Z_{\beta-\alpha,\gamma,N}}{Z_{\beta,\gamma,N}},
  \end{equation*}
  convergence of partition functions and \autoref{lem:energyenstrophy} show that
  \begin{equation*}
  	\expt{e^{\alpha H(\omega_{\beta,\gamma}^N)}}\xrightarrow{N\rightarrow\infty} \EE_{\mu_{\beta,\gamma}}\bra{e^{\alpha \wick{E}(\omega)}}
  \end{equation*}
  with $\omega_{\beta,\gamma}\sim\mu_{\beta,\gamma}$, 
  for any $\alpha$ in a neighbourhood of $0$  ($\beta/\gamma$ as above), and we can conclude by Lévy continuity theorem
  (see \cite[Theorem 4.3]{kallenberg}).
\end{proof}

\section{The Case of the 2-dimensional Sphere}\label{sec:sphere}

Consider the 2-dimensional sphere $\S^2=\set{x\in\R^3:|x|=1}$ as an embedded surface in $\R^3$, its tangent spaces as
subsets of $\R^3$ and gradients of scalar functions as vectors of $\R^3$. On $\S^2$ we consider the uniform measure $d\sigma$
such that $\int_{\S^2}d\sigma=1$. The expressions $x\cdot y,x\times y$ respectively denote in this section the scalar and vector products 
in $\R^3$.

Euler equations on $\S^2$ are given by, for $x\in\S^2$,
\begin{equation*}
	\begin{cases}
	\partial_t \omega(x,t)=x\cdot \pa{\nabla\psi(x,t)\times\nabla\omega(x,t)},\\
	-\Delta \psi(x,t)=\omega(x,t).
	\end{cases}
\end{equation*}
Here $\Delta$ denotes the Laplace-Beltrami operator, and we have to supplement the Poisson equation for the \emph{stream function} $\psi$
with the zero average condition (just as we did on $\T^2$). The Green function of $-\Delta$,
\begin{equation*}
	-\Delta G(x,y)=\delta_y(x)-1
\end{equation*}
has the simple form
\begin{equation*}
	G(x,y)=-\frac1{2\pi} \log |x-y|+c,
\end{equation*}
with $|\cdot|$ the Euclidean distance of $\R^3$ between $x,y\in\S^2$ and $c$ a constant. Just like in the case of flat geometries,
smooth solutions preserve energy and enstrophy (\ref{firstintegrals}). The definition of point vortices dynamics is also
completely analogous to the case on $\T^2$: the vorticity distribution $\omega=\sum_1^N \xi_i \delta_{x_i}$ evolves according to the 
Hamiltonian dynamics (\emph{Helmholtz law})
\begin{equation*}
	\dot x_i=\frac1{2\pi}\sum_{i< j}^N\xi_j \frac{x_j\times x_i}{|x_i-x_j|^2},
\end{equation*}
with Hamiltonian function corresponding to the (renormalised) energy of the configuration,
\begin{equation*}
	H(x_1,\dots x_N)=\sum_{i< j}^N \xi_i\xi_j G(x_i,x_j).
\end{equation*}
We refer to \cite{polvanilorenzodritschel} for a more complete discussion of this setting.

The similarity with the periodic case is such that almost the whole \autoref{sec:clttorus} applies to $\S^2$:
the very same statement of \autoref{thm:clttorus} holds on $\S^2$, with all the involved objects defined as in that case.
The proof proceeds analogously, splitting $G=V_m+W_m$ as in (\ref{potentialsplitting}).
The content of subsections \ref{ssec:gibbsensembles} to \ref{ssec:sinegordon} and \ref{ssec:proofclt} only needs the replacement
of Fourier basis $e_k$ (which we used in Gaussian computations) with spherical harmonics.
In fact, the only argument in the proof of \autoref{thm:clttorus} which needs to be adapted to the case on $\S^2$ is the control on Yukawa partition function
of \autoref{ssec:partitionfunctions}. A careful analysis of the proof of \autoref{prop:zyukawatorus} reveals that
it is sufficient to prove the following bound on $W_m=(m^2-\Delta)^{-1}$.

\begin{rmk}
	The distance between $x,y\in\S^2$ on the surface is given by the angle $\theta\in[0,\pi]$ formed by the vectors $x,y\in\R^3$;
	therefore, by rotation invariance, $G(x,y)=G(\theta)$ and $W_m(x,y)=W_m(\theta)$.
\end{rmk}

\begin{prop}\label{prop:expansionsphere}
	Let $r_m=c\frac{\log m}{m}$ with $c\geq 0$ large enough. It holds, as $m\rightarrow\infty$, uniformly in $\theta\in[0,\pi]$,
	\begin{equation*}
		W_m(\theta)= \pa{-\frac1{2\pi}\log\frac\theta{r_m}+ O(1)}\chi_{\theta\leq r_m}+ O(m^{-2}).
	\end{equation*}
\end{prop}

On $\T^2$, we relied on an explicit representation of $W_m$. Here, we seize the opportunity to present a more robust argument,
based on the well-known representation
\begin{equation}\label{sphereheatrep}
	W_m(x,y)	= \int_0^\infty e^{-m^2 t}p(t,x,y)dt.
\end{equation}
in terms of the heat kernel $p(t,x,y)$. Indeed, the following arguments work more generally on compact Riemannian surfaces without boundary.
We nevertheless prefer to keep using the terminology of $\S^2$, for the sake of simplicity.
We will make use of the following properties of the heat kernel $p(t,x,y)=p(t,\theta)$, for which we refer to \cite{molcanov,nowac}.

\begin{lem}
	It holds, for any $\theta\in[0,2\pi]$,
	\begin{align}
	p(t,\theta)&\leq C,&t\geq1,\label{e:stima1}\\
	p(t,\theta)&\leq \frac C{t\sqrt{\pi-\theta+t}}e^{-\frac{\theta^2}{4t}},&t\leq 1,\label{e:stima2}
	\end{align}
	with $C>0$ independent from $t$. Moreover, for small $t$, uniformly on $\theta$ on compact sets of $[0,\pi)$,
	\begin{equation}
		\label{e:expansion}
		p(t,\theta)= q_t(\theta)H(\theta) + O(1), \quad q_t(\theta)= \frac1{4\pi t}e^{-\frac{\theta^2}{4t}},\quad
		H(\theta)= \frac\theta{\sin\theta}.
	\end{equation} 
\end{lem}

\begin{proof}[Proof of \autoref{prop:expansionsphere}]
	It is not difficult to see, using the estimates (\ref{e:stima1}) and (\ref{e:stima2}), that
	\begin{equation*}
		\int_{r_m^2}^\infty e^{-m^2 t}p(t,\theta)dt+ \chi_{\{\theta\geq r_m\}}\int_0^{r_m^2}e^{-m^2 t}p(t,\theta)dt =O(m^{-2}),
	\end{equation*}
	so we focus on the main term, $\chi_{\theta\leq r_m}\int_0^{r_m^2}e^{-m^2 t}p(t,\theta)dt$. Thanks to \eqref{e:expansion}, we have
	\begin{equation*}
		\int_0^{r_m^2}e^{-m^2 t}p(t,\theta)dt = H(\theta)\int_0^{r_m^2}e^{-m^2 t}q_t(\theta)dt+ O(1).
	\end{equation*}
	Integrating by parts, straightforward computations show that
	\begin{equation*}
			\int_0^{r_m^2}e^{-m^2 t}q_t(\theta)dt=\frac1{4\pi} \int_0^1 \exp\pa{-c^2 \log^2m-\frac{\theta^2}{r_m^2}}\frac{ds}{s}
			 = -\frac1{2\pi}\log \frac{\theta}{r_m}+O(1),
	\end{equation*}
	and since $H(0)=1$ and $H$ is differentiable in $0$, the thesis follows.
\end{proof}

\section{The Case of a Bounded Domain}\label{sec:cltdomain}

In this Section, $D\subset \R^2$ is a bounded domain with smooth boundary, $G(x,y)$ is the Green
function of $-\Delta$ on $D$ with Dirichlet boundary conditions. The naught subscript refers
to boundary conditions: $H^\alpha_0(D)$, $\alpha>0$, are the (fractional) $L^2(D)$-based
Sobolev spaces defined as the closure of compactly supported functions $C^\infty_c(D)$ with respect to the norm
\begin{equation*}
	\norm{u}_{H^\alpha_0(D)}=\norm{(1-\Delta)^{\alpha/2}u}_{L^2(D)},
\end{equation*}
whereas $H^{-\alpha}(D)=H^\alpha_0(D)'$. The Green function $G$ can be represented as the sum of its free version
$G_{\R^2}(x,y)=-\frac{1}{2\pi}\log|x-y|$ and the harmonic extension in $D$ of the values of $G_{\R^2}$ on $\partial D$,
\begin{equation}\label{greendomain}
	G(x,y)=-\frac{1}{2\pi}\log|x-y|+g(x,y), \quad 
	\begin{cases}
	\Delta g(x,y)=0 & x\in D\\
	g(x,y)=\frac{1}{2\pi}\log|x-y| &x\in \partial D	
	\end{cases}
\end{equation}
for all $y\in D$. Both $G$ and $g$ are symmetric, and maximum principle implies that
\begin{equation}\label{harmoniccontbound}
	\frac{1}{2\pi}\log(d(x)\vee d(y))\leq g(x,y)\leq \frac{1}{2\pi} \log\diam(D),
\end{equation}
with $d(x)$ the distance of $x\in D$ from the boundary $\partial D$.

\subsection{Gibbs Ensembles and Gaussian Measures}

The motion of a system of $N$ vortices with intensities $\xi_1,\dots,\xi_N\in\R$ and positions $x_1,\dots,x_N\in D$
is governed by the Hamiltonian function
\begin{equation*}
	H(x_1,\dots,x_n)=\sum_{i< j}^N \xi_i\xi_j G(x_i,x_j)+\frac{1}{2}\sum_{i=1}^{N}\xi_i^2 g(x_i,x_i).
\end{equation*}
The additional (with respect to the cases with no boundary) self-interaction terms involving $g$ are due to the presence of
an impermeable boundary: it is thanks to these terms that the system satisfies (in weak sense) Euler's equations. 
We refer again to \cite[Section 4.1]{marchioropulvirenti} for further details. 
We will consider intensities $\xi_i=\frac{\sigma_i}{\sqrt{\gamma N}}$ with signs $\sigma_i=\pm1$
as in the previous section. We denote by $dx$ the normalized Lebesgue measure on $D$, and for $\gamma>0$, $\beta\geq 0$ we define
\begin{equation}\label{boundsg}
	\nu_{\beta,\gamma,N}(dx_1,\dots,dx_n)= \frac{1}{Z_{\beta,\gamma,N}} \exp\pa{-\beta H(x_1,\dots,x_n)}dx_1,\dots,dx_n.
\end{equation}

\begin{prop}\label{prop:zgreendomain}
	For any choice of $\gamma>0$, $\beta\in\R$, and signs $\sigma_i=\pm 1$, if
  \[
    -8\pi\frac{N}{\max(n_+,n_-)}
      < \frac\beta\gamma
      < 4\pi\frac{N}{1+\min(n_+,n_-)},  
  \]
	then $Z_{\beta,\gamma,N}<\infty$, and the measure $\nu_{\beta,\gamma,N}$ 
  is thus well-defined, where $n_+,n_-$ are, respectively, the number of
  vortices with positive and negative intensity.
\end{prop}

\begin{proof}
  Let us denote by $H_i$ the interaction part and by $H_s$ the self-interaction part of the
  Hamiltonian $H$,
  \begin{equation*}
    H_i
      =\sum_{i< j}^N\xi_i\xi_j G(x_i,x_j), \quad H_s=\frac12\sum_{i=1}^N\xi_i^2g(x_i,x_i).
  \end{equation*}
    
  If $\beta<0$, $-\beta H_s$ is bounded from above by \eqref{harmoniccontbound}.
  Since $G\geq0$, we can neglect in $H_i$ the contribution of vortices with different
  sign and
  \[
    \beta H_i
      \leq -\frac{\beta}{2\gamma N}\sum_{\sigma_i,\sigma_j>0}G(x_i,x_j)
        -\frac{\beta}{2\gamma N}\sum_{\sigma_i,\sigma_j<0}G(x_i,x_j)
      \vcentcolon= -\beta H_i^+ - \beta H_i^-
  \]
  The terms $H_i^+$, $H_i^-$ are functions on disjoint sets of variables, so the
  integral of their exponential factorizes in the product of two integrals. We analyse
  the first integral, the estimate of the second will follow likewise. Let
  $I_+=\{i:\sigma_i>0\}$. Again by
  \eqref{harmoniccontbound}, the self-interaction terms is bounded, therefore
  \[
    \int {D^{i_+}} e^{-\beta H_i^+}
      \lesssim \int_{D^{i_+}}\prod_{i\in I_+}\prod_{j\in I_+,j\neq i}|x_i-x_j|^{\frac\beta{4\pi\gamma N}}
      \leq \prod_{i\in I_+}\Bigl(\int_D \,dx_i\prod_{j\in I_+,j\neq i}
        \int_D |x_i-x_j|^{\frac\beta{4\pi\gamma N}n_+}\,dx_j\Bigr)^{\frac1{n_+}}
  \]
  The integrals above are finite if $\frac\beta{4\pi\gamma N}n_+>-2$. Likewise,
  for $H_i^-$ we obtain $\frac\beta{4\pi\gamma N}n_->-2$.
  
  We turn to the case $\beta>0$. By the H\"older inequality with conjugate exponents
  $p$ and $q$, we can bound separately the contributions of $H_i$ and $H_s$ 

	Thanks to \eqref{harmoniccontbound}, it holds
	\begin{equation*}
		\int_{D^N}e^{-\beta qH_s(x_1,\dots,x_N)}dx_1\dots dx_N \leq \pa{\int_D d(x)^{-\frac{\beta q}{4\pi\gamma N}}dx}^N <\infty
	\end{equation*}
	as soon as $\frac{\beta}{4\pi\gamma}<\frac Nq$. As for the interaction term,
  since $G$ is positive and $g$ is uniformly bounded from above,
	\begin{equation*}
		-p\beta H_i
       \leq -\frac{\beta p}{2\pi\gamma N}\sum_{\sigma_i\cdot\sigma_j<0}^N\log|x_i-x_j| + C N.
	\end{equation*}
  Assume without loss of generality that $n_-\leq n_+$, then by the H\"older inequality,
	\begin{align*}
		\int_{D^N}e^{-\beta H_i}
		  &\lesssim \int_{D^{n_-}}\prod_{i\in I_+}\Bigl(
        \int_D \prod_{j\in I_-}|x_i-x_j|^{-\frac{p\beta}{2\pi\gamma N}}
        \,dx_i\Bigr)\\
		  &= \int_{D^{n_-}}\Bigl(\int_D \prod_{j\in I_-}|y-x_j|^{-\frac{p\beta}{2\pi\gamma N}}
        \,dy\Bigr)^{n_+}\\
		  &\leq \int_{D^{n_-}}\prod_{j\in I_-}\Bigl(\int_D
        |y-x_j|^{-\frac{p\beta}{2\pi\gamma N}n_-}\,dy\Bigr)^{\frac{n_+}{n_-}}.
	\end{align*}
	The right-hand side is finite if $\frac{p\beta}{2\pi\gamma}n_-<2$.
  Combining the two conditions on $p,q$ we get the announced restriction on
	$\beta/\gamma$.
\end{proof}

The reader will notice that, unlike in \autoref{prop:zgreentorus}, when $N\rightarrow\infty$
we still have a restriction on the values of $\beta/\gamma$. See \autoref{rmk:restriction}
for more details.

We define the probability $\mu_{\beta,\gamma}^N$ on finite signed measures $\M(D)$ as the law of
\begin{equation*}
\omega_{\beta,\gamma}^N=\sum_{i=1}^N \xi_i \delta_{x_i},
\end{equation*}
with $x_1,\dots x_n$ sampled under $\nu_{\beta,\gamma,N}$.
In the case of a bounded domain we will assume the \emph{neutrality condition}
\begin{equation}\label{neutral}
  \sum_{i=1}^N \sigma_i
    = 0,
\end{equation}
so that $\omega_{\beta,\gamma}^N$ has zero average.

The limiting Gaussian random field should also have zero space average. Since the constant function $1$
does not belong to the spaces in which we set the problem (it does not satisfy the Dirichlet b.c.),
the definition is somewhat more involved than it was on $\T^2$. Define the bounded linear operator
\begin{equation*}
	M:L^2(D)\rightarrow L^2(D), \quad Mf(x)=f(x)-\int_{D} f(y)dy.
\end{equation*}
For $\gamma>0$ and $\beta\geq 0$, let $\omega_{\beta,\gamma}$ be the centred Gaussian random field on $D$ with covariance
\begin{equation*}
\forall f,g\in L^2(D), \quad 	
\expt{\brak{\omega_{\beta,\gamma},f}\brak{\omega_{\beta,\gamma},g}}=\brak{f,Q_{\beta,\gamma}g}, 
\quad Q_{\beta,\gamma}=M^*(\gamma-\beta\Delta)^{-1}M.
\end{equation*}
Equivalently, $\omega_{\beta,\gamma}$ is a centred Gaussian stochastic process indexed by $L^2(D)$ with the 
specified covariance. Analogously to the torus case, $\omega_{\beta,\gamma}$ can be identified with a random distribution taking values in
$H^s(D)$ for all $s<-1$. 

Renormalised energy of the vorticity distribution $\mu_{\beta,\gamma}$ is defined just as in (\ref{normenergy}),
and the equivalent definition of $\mu_{\beta,\gamma}$ provided by \autoref{lem:energyenstrophy} still applies in this context.
In fact, all Gaussian computations in Fourier series of the last Section still work on domains $D\subset \R^2$ 
if one considers an orthonormal basis of $L^2(D)$ diagonalising the Laplace operator: for $n\in\N$,
\begin{equation*}
-\Delta e_n=\lambda_n e_n, \quad \lambda_n\sim n,
\end{equation*}
the latter being the well known Weyl's law. The main difference is that explicit expression in Fourier series on $D$ are complicated by the
presence of the zero-averaging operator $M$ in the covariance.
We are now able to state the main result of the Section, a perfect analogue of the Central Limit Theorem we proved above on $\T^2$.
\begin{thm}\label{thm:cltdomain}
	Let $\beta/\gamma\in [0,8\pi)$, assume the neutrality condition \eqref{neutral},
  and set $\bar g=\int_{D} g(y,y)dy$. It holds:
	\begin{enumerate}
		\item[(1)] $\lim_{N\rightarrow\infty} Z_{\beta,\gamma,N}=e^{\beta \bar g}Z_{\beta,\gamma}$;
		\item[(2)] the sequence of $\M$-valued random variables $\omega^N\sim\mu_{\beta,\gamma}^N$ 
		converges in law on $H^s(D)$, any $s<-1$, to a random distribution $\omega\sim \mu_{\beta,\gamma}$,
		as $N\rightarrow\infty$;
		\item[(3)] the sequence of real random variables $H(\omega^N)-\bar g$ converges in law to $\wick{E}(\omega)$ as 
		$N\rightarrow\infty$, with $\omega^N,\omega$ as in point (2).
	\end{enumerate}
\end{thm}

\begin{rmk}
	Minor modifications of our arguments allow to replace the neutrality condition on intensities 
	with the hypothesis $\sum_{i=1}^N \xi_i=o((\log N)^{-1/2})$.
	Moreover, it is possible to consider random signs $\sigma_i$ taking values $\pm1$ with probability $1/2$, or more generally
	i.i.d. bounded signs with zero expected value. Such generalisations are in fact inessential from the physical point of view,
	namely we are still dealing with fluctuations around a null profile (see \autoref{sec:meanfield}): we omit details.
\end{rmk}

We conclude this paragraph proving the case $\beta=0$ (and $\gamma=1$, for notational simplicity): 
if we can then provide a uniform bound for partition functions $Z_{\beta,\gamma,N}$,
the content of \autoref{ssec:proofclt} completely carries on to the domain case. In the remainder of this Section we show out how to adapt the strategy
we used in the torus case to control partition functions.

The expression (\ref{doubleintegralsmooth}) of double stochastic integrals with respect to white noise still holds,
and so does \autoref{lem:vorticesisometry} in the following form:
\begin{lem}
	Let $\omega^N\sim \mu_{0,\gamma}^N$.
	On continuous functions $h\in C(D^2)$ vanishing on the diagonal, 
	\emph{i.e.} $h(x,x)=0$ for all $x$, define the map
	\begin{equation*}
	h\mapsto \int_{D^2}h(x,y)d\omega^N(x)d\omega^N(y)=\sum_{i\neq j}\xi_i\xi_j h(x_i,x_j).
	\end{equation*}
	Since it holds
	\begin{equation*}
	\expt{\pa{\sum_{i\neq j}\xi_i\xi_j h(x_i,x_j)}^2}\leq C_\gamma \norm{h}^2_{L^2(D^2)}
	\end{equation*}
	with $C_\gamma$ a constant independent of $N$, the map takes values in $L^2(\mu_{0,1}^N)$, and it extends by density to a bounded linear map from $\dot L^2(D^2)$ to $L^2(\mu_{0,1}^N)$
  which we will denote by
	\begin{equation*}
	  f\mapsto \int_{D^2}f(x,y)\wick{d\omega^N(x)d\omega^N(y)}.
	\end{equation*}
\end{lem}

The proof only differs from the one on $\T^2$ in that is uses neutrality of total intensity in place of the zero average condition.
In considering the relation between the Hamiltonian and renormalised energy, another relevant difference with respect to the
torus case appears: defining the renormalised energy of point vortices as in \autoref{sec:clttorus},
\begin{align*}
	2\wick{E}&=\int_{D^2} G(x,y)\wick{d\omega^N\otimes d\omega^N}=\sum_{i\neq j}\xi_i\xi_j G(x_i,x_j)\\
	&=2H-\sum_{i=1}^N \xi_i^2 g(x_i,x_i).
\end{align*}
This is why we need corrections depending on $\bar g=\int_{D}g(y,y)dy$ in points (1) and (3) of \autoref{thm:cltdomain}: 
the Hamiltonian $H$ alone is not a centred variable, and its mean value is
\begin{equation*}
	\sum_{i=1}^N \xi_i^2 g(x_i,x_i)=\frac1{N} \sum_{i=1}^N g(x_i,x_i),
\end{equation*}
which converges by the law of large numbers to $\bar g$.
That being said, proceeding as in \autoref{ssec:proofclt} straightforwardly concludes the proof of the case $\beta=0$.

\subsection{Potential Splitting on Bounded Domains}

We want to decompose $G=V_m+W_m$ as in \autoref{sec:clttorus}, with $V_m$ a regular (\emph{long range}) potential converging to $G$
as $m\rightarrow\infty$, and $W_m$ a singular but vanishing remainder. In order for our strategy to work we need to rewrite the part of
$H$ corresponding to $V_m$ as sum of covariances (in particular, positive terms) of a regular Gaussian field with zero space average.
At the same time, we will need a quite precise description of $W_m$. We thus choose $W_m$ as the Green function of $m^2-\Delta$ on $D$
with Dirichlet boundary conditions, that is
\begin{equation}\label{yukawapotentialdomain}
W_m(x,y)=\frac{1}{2\pi}K_0(m|x-y|)+w_m(x,y), \quad 
\begin{cases}
(m^2-\Delta) w_m(x,y)=0 & x\in D\\
w_m(x,y)=-\frac{1}{2\pi}K_0(m|x-y|) &x\in \partial D	
\end{cases}
\end{equation}
for all $y\in D$, and where we notice that $\frac{1}{2\pi}K_0(m|x-y|)=W_{m,\R^2}(x,y)$ is the Green function of $m^2-\Delta$ on the whole plane.
We then set
\begin{equation*}
	V_m=G-W_m, \quad v_m=g-w_m.
\end{equation*}
Unfortunately, $V_m$ is not zero averaged, so we need to further define the potential
\begin{equation}\label{definitionvzero}
	V_m^0(x,y)=V_m(x,y)-\int_D V_m(x,y)dy-\int_D V_m(x,y)dx+\int_{D^2} V_m(x,y)dxdy,
\end{equation}
which we will use as covariance kernel for the Gaussian field $F_m$: indeed, notice that, as an integral kernel,
\begin{equation*}
	V_m^0=M^*m^2(-\Delta(m^2-\Delta))^{-1}M,
\end{equation*}
thus $V_m^0$ is positive definite and zero averaged. 

Looking now at the corresponding decomposition of the Hamiltonian,
\begin{align*}
	H&=\sum_{i<j}^N \xi_i\xi_j W_m(x_i,x_j)+\frac{1}{2}\sum_{i=1}^{N}\xi_i^2 w_m(x_i,x_i)
	+\sum_{i<j}^N \xi_i\xi_j V_m(x_i,x_j)+\frac{1}{2}\sum_{i=1}^{N}\xi_i^2 v_m(x_i,x_i)\\
	&:=H_{W_m}+H_{V_m},
\end{align*}
a simple computation exploiting the neutrality condition yields
\begin{equation*}
	\sum_{i,j}^N \xi_i\xi_j V_m(x_i,x_j)=\sum_{i,j}^N \xi_i\xi_j V_m^0(x_i,x_j)-\sum_{i=1}^{N}\xi_i^2 V_m(x_i,x_i),
\end{equation*}
so that, since $V_m+v_m=V_{m,\R^2}$ (the Green function of $-m^{-2}\Delta(m^2-\Delta)$), we can rewrite
\begin{equation*}
	H_{V_m}=\frac{1}{2}\sum_{i,j}^N \xi_i\xi_j V_m^0(x_i,x_j)-\frac{1}{2}\sum_{i=1}^{N}\xi_i^2 V_{m,\R^2}(x_i,x_i).
\end{equation*}
One can easily show that $V_{m,\R^2}$ is a regular, symmetric, translation invariant function; moreover,
it has a global maximum in $V_{m,\R^2}(0,0)=\frac{1}{2\pi}\log m+o(\log m)$, 
as it is shown by taking the difference of
\begin{equation*}
	G_{\R^2}(x,y)=-\frac{1}{2\pi}\log |x-y|, \quad W_{m,\R^2}(x,y)=\frac{1}{2\pi}K_0(m|x-y|)\sim -\frac{1}{2\pi}\log(m|x-y|),
\end{equation*}
for close $x,y\in\R^2$.
This, together with (\ref{definitionvzero}), implies that for all $x\in D$ we also have $V_m^0(x,x)=\frac{1}{2\pi}\log m+o(\log m)$ .

\begin{lem}\label{lem:Fboundsdomain}
	Let $F_m$ be the centred Gaussian field on $D$ with covariance kernel $V_m^0$. 
	There exists a version of $F_m(x)$ which is $\alpha$-H\"older for all $\alpha<1/2$, and moreover
	for any $\alpha>0$, $p\geq 1$ and $m\rightarrow \infty$, it holds
	\begin{align}
	\label{Fmomentsdomain}
	\expt{\norm{F_m}_p^p}&\simeq_p (\log m)^{p/2}\\
	\label{Fexpmomentsdomain}
	\expt{\exp\pa{-\alpha\norm{F_m}_2^2}}&\lesssim m^{-\frac{\alpha}{2\pi}}.
	\end{align}
\end{lem}

\begin{proof}
	H\"older property descends from Kolmogorov continuity theorem since $V_m$ is continuously differentiable (and so is $V_m^0$).
	The estimate of $p$-moments is the same as in the periodic case, so let us turn to exponential moments.
	Identifying kernels and their associated integral operators, it holds
	\begin{equation*}
		\expt{\exp\pa{-\alpha \norm{F_m}_2^2}}
		=\exp \set{-\frac{1}{2}\trace\pa{\log\pa{1+2\alpha V_m^0}}}.
	\end{equation*}
	Hence, we only need to compute the asymptotic behaviour in $m$ of $\trace V_m^0$, since then we can apply
	the inequalities $x-\frac{x^2}{2}<\log(1+x)<x$ and conclude as in \autoref{lem:Fboundstorus}.
	We resort again to Fourier series: by definition of the kernel $V_m^0$ we have
	\begin{align*}
		\trace V_m^0 &= \sum_{n=1}^{\infty}\int_{D^2} V_m^0(x,y)e_n(x)e_n(y)dxdy\\
		&=\trace V_m -2 \sum_{n=1}^{\infty} \bar e_n \int_{D^2} V_m(x,y)e_n(x)dxdy+\int_{D^2} V_m(x,y)dxdy\sum_{n=1}^{\infty} \bar e_n^2\\
		&=\trace V_m-\sum_{n=1}^{\infty} \frac{m^2 \bar e_n^2}{\lambda_n(m^2+\lambda_n)}= \trace V_m+ O(1), \quad m\rightarrow\infty,
	\end{align*}
	where we denoted $\bar e_n$ the space averages of $e_n(x)$ (that is, the Fourier coefficients of the constant function 1).
	The last passage is a consequence of
	\begin{align*}
		0\leq \sum_{n=1}^{\infty} \frac{m^2 \bar e_n^2}{\lambda_n(m^2+\lambda_n)}
		&\leq \pa{\sum_{n=1}^{\infty} \frac{m^4}{\lambda_n^2(m^2+\lambda_n)^2}}^{1/2}
		\lesssim \pa{\int_1^\infty \frac{m^4}{x^2(m^2+x)^2}dx}^{1/2}\\
		&=\pa{\frac{m^2+2}{m^2+1}-\frac{2 \log(m^2+1)}{m^2}}^{1/2}=O(1), \quad m\rightarrow\infty,
	\end{align*}
	where we used $\pa{\sum_{n=1}^{\infty}\bar e_n^4}^{1/2}\leq \sum_{n=1}^{\infty}\bar e_n^2=\norm{1}^2_{L^2(D)}=1$ and
	Cauchy-Schwarz inequality. We conclude by noting that
	\begin{equation*}
		\trace V_m=\sum_{n=1}^{\infty} \frac{m^2}{\lambda_n(m^2+\lambda_n)}=V_m(0,0).\qedhere
	\end{equation*}
\end{proof}

We can now apply the transformation
\begin{equation*}
	e^{-\beta H_{V_m}}=e^{\frac{\beta}{2\gamma}V_{m,\R^2}(0,0)} \expt{e^{\imm \sqrt{\beta} \sum_{i=1}^N \xi_i F_m(x_i)}}
\end{equation*}
and proceed as in the previous Section. The proof of \autoref{prop:Fcomplexboundstorus} in the bounded domain setting
is just the same, thanks to \autoref{lem:Fboundsdomain}. We are only left to prove the analogue of \autoref{prop:zyukawatorus},
from which a uniform bound on partition functions is derived as in \autoref{cor:zuniformbound}.
 
\begin{prop}\label{prop:zyukawadomain}
	Let $N\geq 1$, $|\beta/\gamma|\leq 8\pi$ and $m>0$. There exists a constant $C_{\beta,\gamma}>0$ such that
	\begin{equation*}
	\int_{\T^{2N}} e^{-\beta H_{W_m}}dx_1\cdots dx_n \leq \pa{1+C_{\beta,\gamma}\frac{(\log m)^2}{m^2}}^N.
	\end{equation*}
\end{prop}

\begin{proof}
	As in the first part of the proof of \autoref{prop:zyukawatorus}, we reduce by means of H\"older inequality to bound the integral
	\begin{equation*}
	 I=\int_{D^2} e^{\frac{\beta}{2\gamma}W_m(x,d)}dxdy.
	\end{equation*}
	We thus proceed to bound pointwise the interaction potential $W_m(x,y)=W_{m,\R^2}(x,y)+w_m(x,y)$.
	Let us first fix $x$, and consider the small radius $r_m=\frac{2\log m}{m}$, as we did in \autoref{prop:zyukawatorus}.
	For $m$ large enough, $B(x,r_m)\subseteq D$, and we have showed in \autoref{sec:clttorus} that for all $x,y\in \R^2$,
	\begin{equation*}
		W_{m,\R^2}(x)\leq -\frac{1}{2\pi}\log\pa{\frac{|x-y|}{r_m}}\chi_{B(x,r_m)}(y)+ \frac{C}{m^2}.
	\end{equation*}
	We are thus left to bound $w_m(x,y)$: by definition (\ref{yukawapotentialdomain}) and the maximum principle, it holds,
	for all $x$ uniformly in $y$,
	\begin{equation*}
		w_m(x,y)\leq \frac{1}{2\pi}K_0(m d(x)) \leq -\frac{1}{2\pi}\log\pa{\frac{d(x)}{r_m}}\chi_{d(x)<r_m}+ \frac{C}{m^2}.
	\end{equation*}
	Going back to $I$, we get
	\begin{align*}
	I&\leq e^{C/m^2}\int_{B(x,r_m)} \pa{1+\pa{\frac{|x-y|}{r_m}}^{-\frac{\beta}{4\pi\gamma}}}dy \cdot
	                \int_D \pa{1+\pa{\frac{d(x)}{r_m}}^{-\frac{\beta}{4\pi\gamma}}}dx\\
	&\leq e^{C/m^2} \pa{1+C r_m^2}^2,
	\end{align*}
	which concludes just as in \autoref{prop:zyukawatorus}.	 
\end{proof}

\begin{rmk}\label{rmk:restriction}
  The technical reason behind the parameter restriction in \autoref{prop:zyukawadomain} and \autoref{prop:zgreendomain} above could be avoided if a local decomposition of the
  Yukawa potential as in \autoref{prop:expansionsphere} is available for a general domain
  $D$ with smooth enough boundary. Indeed, in that case, one could deduce that
  $Z_{\beta,\gamma,N}<\infty$, and thus that the meaure $\nu_{\beta,\gamma,N}$ is well
  defined for all values of $\beta>0$, $\gamma>0$. Likewise, \autoref{prop:zyukawadomain}
  and in turns \autoref{prop:zgreendomain} would hold woithout restrictions.
  
  A way to prove a local decomposition for the Yukawa potential is to use the same
  strategy of \autoref{sec:sphere}, namely the general representation \eqref{sphereheatrep},
  that holds beyond the geometry of the sphere. Through the point of view of the heat
  kernel, the role of the geometry of the domain and of its boundary becomes apparent
  in terms of the divergence in time of the heat kernel, whose behaviour depends on the
  number of geodetics and their intersection with the boundary. We refer to
  the fundamental \cite{molcanov} for further details. We notice in particular that
  if the intrinsic geometry of the domain is geodesically convex, that in the flat
  metric means that the domain is convex, the same estimates, in particular
  \cite[Theorem 2.1]{molcanov}, of the case without boundary such as the sphere or the
  torus, hold. This justify the following corollary, that fully generalizes the
  central limit theorem of \cite{bodineau} from the sphere to general convex domains.
\end{rmk}

\begin{cor}
  Assume the neutrality condition \eqref{neutral}.
  If $D$ is a convex domain, then the conclusions of \autoref{thm:cltdomain}
  hold for all $\beta>0$ and $\gamma>0$.
\end{cor}

\section{Concluding Remarks: a Comparison with Mean Field Theory}\label{sec:meanfield}

In this Section we reinterpret our results in sight of the mean field limit studied by \cite{clmp92,clmp95,Kie1993}
(see also \cite{lionsbook}). Those works cover the case of vortices with identical intensities, while \cite{bodineau, neri} consider vortices
with (random) intensities of different signs. Vortices with random intensities on $\mathbb{S}^2$ have been analyzed in \cite{KieWan2012}.
We also mention results on vorticity filaments in dimension 3, \cite{michele17,michele19}.

The scaling of intensities $|\xi|\sim N^{-1}$, is dictated by energy considerations, in order for the dominant (infinite) self-interaction
term to vanish. It is \emph{not} the scaling we assumed in the previous Sections, as it corresponds to the law of large number scaling.
The scaling of inverse temperature $\beta\sim N$ is chosen so that the limit is non-trivial, see \cite{marchioropulvirenti}.
The resulting Hamiltonian on a bounded domain $D\subset \R^2$, with parameters of order one up to rescaling, is
\begin{equation*}
	\frac1{N}\sum_{i< j}\sigma_i\sigma_j G(x_i,x_j)
	+ \frac1{2N}\sum_{i=1}^N\sigma_i^2 g(x_i,x_i),
\end{equation*}
with $\sigma_i$ uniformly bounded. The corresponding Gibbs measure coincides with our $\nu_{\beta}^N=\nu_{\beta,1}^N$.

In the case of a bounded domain, for vortices with the same intensity, \cite{clmp92} proved that the single vortex distribution,
that is the one dimensional marginal of $\nu_{\beta}^N$, converges to a superposition of solutions to the Mean Field Equation,
\begin{equation}\label{mfe}
  \omega= \frac{e^{-\beta\psi}}{\int_D e^{-\beta\psi}dx}, \quad -\Delta\psi = \omega,
\end{equation}
with the Poisson equation for the stream function $\psi$ being complemented with Dirichlet boundary conditions.
Solutions to (\ref{mfe}) are particular steady solutions of the Euler equations that minimize the energy-entropy
functional $\beta E+S$ (defined in (\ref{firstintegrals})).
A unique minimum exists when $\beta>0$ (and for $\beta\leq 0$ close enough to $0$), so that $\nu_{\beta}^N$
converges, in the sense of finite dimensional distributions, to an infinite product measure (\emph{propagation of chaos}).
Connections of the mean field equation with the microcanonical ensemble and equivalence with the canonical
ensemble are considered in \cite{clmp95}.

The case of intensities with different signs is studied in \cite{bodineau} through a large deviations approach.
Under the assumption that the empirical measure of intensities converges to a probability distribution $\mu$,
the joint empirical measure of intensities and positions satisfies a large deviation principle with speed
$N^{-1}$, and the extended energy-entropy functional as rate function:
\begin{equation}\label{freeenergy}
  H(\nu)+ \frac{\beta}{2}\int_{\R^2\times D^2} \sigma\sigma' G(x,x')\nu(d\sigma,dx)\nu(d\sigma',dx'),
\end{equation}
where $H$ is the relative entropy of $\nu$ with respect to the product of $\mu$ and the normalized Lebesgue measure on $D$.
The mean field equation satisfied by the density (corresponding to the Euler-Lagrange equation for the minimisation problem of the rate function) is
\begin{equation*}
	\rho(\sigma,x)	= \frac{1}{Z}e^{-\beta\sigma\psi},
\end{equation*}
with $Z$ a normalising constant and $\psi$ is the averaged stream function,
\begin{equation}\label{stream}
  \psi(x)= \int \sigma G(x,y)\rho(\sigma,y)\mu(d\sigma)dy.
\end{equation}
Similar statement also hold in the periodic case.

Looking back to our setting, in both the case of zero average vortices on $\T^2,\S^2$, and the one of vortices in a bounded domain
$D$ with neutral global intensity, for $\beta\geq 0$, the free energy (\ref{freeenergy}) is non-negative and attains the value zero on
the $N$-fold product uniform measure. Moreover, the stream function (\ref{stream}) is null. 
The large deviations principle of \cite{bodineau} implies a law of large numbers, while our \autoref{thm:clttorus} and \autoref{thm:cltdomain}
provide the convergence of fluctuations with respect to the null average. We mention again the central limit theorem derived in \cite{bodineau},
which is however restricted to a disk domain and to a small class of test function.

\bibliographystyle{plain}

\end{document}